\documentclass{amsart}
\usepackage{amssymb}
\usepackage{amsmath}
\usepackage{amsfonts}
\usepackage{chngcntr}
\newtheorem{theorem}{Theorem}[section]

\newtheorem{prop}[theorem]{Proposition}

\newtheorem{remark}[theorem]{Remark}

\newtheorem{innercustomthm}{Proposition}
\newenvironment{customthm}[1]
  {\renewcommand\theinnercustomthm{#1}\innercustomthm}
  {\endinnercustomthm}

\renewcommand{\l}{\lambda}
\newcommand{\R}{\mathbb R}

\newcommand{\p}{\partial}
\newcommand{\al}{\alpha}

\newcommand{\tA}{{\tilde{A}}}

\begin{document}
\title[Large data global regularity for the Skyrme model]{Large data global regularity for the classical equivariant  Skyrme model}

\author{Dan-Andrei Geba and Manoussos G. Grillakis}

\address{Department of Mathematics, University of Rochester, Rochester, NY 14627, U.S.A.}
\email{dangeba@math.rochester.edu}
\address{Department of Mathematics, University of Maryland, College Park, MD 20742, U.S.A.}
\email{mggrlk@math.umd.edu}

\date{}

\begin{abstract}
This article is concerned with the large data global regularity for the equivariant case of the classical Skyrme model and proves that this is valid for initial data in $H^s\times H^{s-1}(\R^3)$ with $s>7/2$.
\end{abstract}


\maketitle


\section{Introduction}


\subsection{Statement of the problem and main result} 

One of the fundamental models in classical field theory is the \textit{Gell-Mann-L\'{e}vy model} \cite{GL}, also known as the \textit{classical nonlinear $\sigma$ model}, which is described by the action
\begin{equation*}
S\,=\,\int_{\R^{3+1}} \left\{ \frac 12 S^{\mu}_{\mu}\right\}\,dg,
\end{equation*}
where  $S_{\mu\nu}\,=\,h_{AB}\,\partial_\mu U^A \,\partial_\nu U^B$ is the pullback metric associated to a map $U: \R^{3+1} \to \mathbb{S}^3$,  $g=\text{diag}(-1,1,1,1)$  is the Minkowski metric, and $h$ is the induced Riemannian metric on $\mathbb{S}^3$ from the Euclidean one on $\mathbb{R}^4$. An important feature of this theory is that it does not admit \textit{static topological solitons} and this is precisely the motivation that led Tony Skyrme to introduce, in a series of seminal papers \cite{S1, S2, S3}, a physically relevant modification of the Gell-Mann-L\'{e}vy model, which no longer has this limitation. The static topological solitons of the Skyrme theory are known as \textit{skyrmions} and, from a historical perspective, they represent the first of their type to model a particle. The action corresponding to the \textit{Skyrme model} is given by
\begin{equation}
S\,=\,\int_{\R^{3+1}}  \left\{\frac 12 S^{\mu}_{\mu}+\, \frac{\alpha^2}{4} (S^{\mu}_{\mu}S^{\nu}_{\nu}-S^{\mu\nu} S_{\mu\nu})\right\}\,dg,
\label{S}
\end{equation} 
where $\al$ is a constant having the dimension of length. As the value of $\alpha$ does not play an important role in our arguments other than being positive, from here on in, we set it to $\alpha = 1$ in order to simplify the exposition. For a more comprehensive discussion of the physical descriptions and motivations for both the Gell-Mann-L\'{e}vy and Skyrme models, we refer the reader to our recent monograph \cite{GGr} and references therein.

Our focus in this article is on the Euler-Lagrange equations associated to the equivariant case of the Skyrme model; i.e., we work with formal critical points for \eqref{S} under the ansatz
\[
U(t,r,\omega)= (u(t,r),\omega),
\]
where
\[
g\,=\,-dt^2 + dr^2 + r^2 d\omega^2, \qquad h\,=\,du^2 + \sin^2 u\, d\omega^2,\]
are the previous metrics written in polar form. The relevant variational equation is the one for the azimuthal angle $u$,
\begin{equation}
\left(1+\frac{2\sin^2u}{r^2}\right)(u_{tt}-u_{rr})-\frac{2}{r}u_r+\frac{\sin2u}{r^2}\left(1+u_t^2-u_r^2+\frac{\sin^2u}{r^2}\right)=0, \label{main}
\end{equation}
which is of quasilinear type. A formal computation shows that solutions to this equation have the energy-type quantity
\begin{equation}
E[u](t)=\int_0^\infty \left\{\left(1 + \frac{2\sin ^2 u}{r^2}\right)\frac{u_t^2+u_r^2}{2}+
\frac{\sin ^2 u}{r^2} +\frac{\sin^4 u}{2r^4}
\right\}\,r^2 dr \label{tote}
\end{equation}
conserved in time. We are studying finite energy solutions, which necessarily satisfy 
\[
u(t,0)\,\equiv\, u(t,\infty)\,\equiv\,0\,(\text{mod}\,\pi).\]
The integer
\[
\frac{u(t,\infty)-u(t,0)}{\pi}
\]
is called the \textit{topological charge} of the map $U$ and, like the energy, it is also a conserved quantity. In what follows, we assume that 
\begin{equation}
u(t,0)=N_1\pi, \quad N_1\in \mathbb{N}, \qquad u(t,\infty)=0.
\label{bdry}
\end{equation}

The subsequent theorem is the main result of this paper. 
\begin{theorem}
Let $(u_0,u_1)$ be radial initial data with
\[
(u_0,u_1)\in H^s\times H^{s-1}(\R^3), \quad s>\frac{7}{2},
\]
which meets the compatibility conditions
 \[u_0(0)=N_1\pi, \qquad u_0(\infty)=u_1(0)=u_1(\infty)=0.\] 
Then there exists a global radial solution $u$ to the Cauchy problem associated to \eqref{main} with $(u(0),u_t(0))=(u_0,u_1)$, satisfying \eqref{bdry} and  
\[
u\in C([0,T], H^s(\R^3))\cap C^1([0,T], H^{s-1}(\R^3)), \qquad (\forall)\,T>0.
\] 
\label{main-th}
\end{theorem}

\begin{remark}
It is important to compare this result with what is known for the corresponding equation of the Gell-Mann-L\'{e}vy model, i.e.,
\begin{equation*}
u_{tt}-u_{rr}-\frac{2}{r}u_r+\frac{\sin2u}{r^2}=0,
\end{equation*}
which is of semilinear type and looks considerably simpler than \eqref{main}. Surprisingly, Shatah \cite{S88} showed that there are smooth data that lead to solutions of this equation that blow up in finite time, with Turok and Spergel \cite{PhysRevLett.64.2736} later finding such solutions in closed form. 
\end{remark}

\subsection{Comments on previous relevant works and comparison to main result} 
Since it first appeared, the Skyrme model has received considerable interest from both the mathematics and physics communities, with comprehensive lists of references being available in the book by Manton and Sutcliffe \cite{MS} and in our monograph \cite{GGr}. In here, we strictly focus on a number of works that better correlate to ours. 

For the static problem, we start by mentioning the proof for the existence and uniqueness of skyrmions with arbitrary topologic charges, which is due to Kapitanski\u{i} and Ladyzhenskaya \cite{KL} and is based on variational techniques. An alternative approach relying strictly on ODE type methods was formulated later by McLeod and Troy \cite{LT}. The asymptotic stability of the skyrmion with unit topologic charge was numerically investigated by Bizo\'{n}, Chmaj, and Rostworowski \cite{BCR07}, while, just recently, Creek, Donninger, Schlag, and Snelson \cite{CDSS16} proved rigorously its linear stability.

In what concerns the non-static problem, the Euler-Lagrange equations, for both the general and equivariant cases of the Skyrme model, have been the focus of quite a few studies in recent years. As these are evolution equations, one natural question to study about them is the well-posedness of the associated initial value problem. This is a very challenging task, mainly due to the quasilinear nature of the equation (displayed above by \eqref{main}) and the fact that scaling heuristics in a small energy scenario show that the Cauchy problem is supercritical with respect to the energy (for more details, see \cite{GNR}). 

For the general case, Wong \cite{W11} displayed regimes in which the problem is regularly hyperbolic and, consequently, is locally well-posed for almost stationary initial data. The same paper also showcased frameworks which lead to an ultrahyperbolic-type breakdown of hyperbolicity and, thus, to ill-posedness. In the equivariant case, Geba, Nakanishi, and Rajeev \cite{GNR} proved global well-posedness and scattering for the Cauchy problem associated to \eqref{main}, when $N_1=0$ and the initial data have a small Besov-Sobolev norm at the level of $H^{5/2}(\R^3)$. 

However, the most relevant work to the current paper, which also served as one of its sources of inspiration, is due to Li \cite{Li}. His main result is a proof of Theorem \ref{main-th} in the more restrictive setting $s\geq 4$. Our argument follows the general framework of \cite{Li}; nevertheless, certain key parts have been reworked and streamlined and the main argument is now transparent (e.g., Propositions \ref{prop-Phit-h1} and \ref{prop-Phitt-h1}). Additionally, we have a comprehensive appendix that discusses the regularity of the initial data, which is an important element of the proof. A difference of the present work with \cite{Li} is in the last step of our argument, where our approach is able to handle fractional derivatives. A final remark has to do with our belief that Theorem \ref{main-th}  is optimal with regard to the spaces used in its proof. However, we do not pursue this issue in the present work.

A parallel literature exists on other Skyrme-like theories, like the Faddeev and Adkins-Nappi models, for which we ask the interested reader to consult, yet again, our book \cite{GGr}.  


\subsection{Outline of the paper}

In section 2, we reformulate our result in terms of a newly introduced function $v$ and reduce the argument for Theorem \ref{main-th} to the verification of a continuation criterion for $v$. Subsequently, we construct another auxiliary function $\Phi$, which is intimately tied to $v$, but is more amenable to the methods we have in mind. Also here, we reduce once more the argument to the proof of the finiteness of certain Sobolev norms for derivatives of $\Phi$. Section 3 is devoted to setting up the necessary notational conventions and to collecting the analytic tools used throughout the paper. In section 4, we start in earnest our proof and show that both $v$ and $\Phi$ satisfy energy-type estimates, which lead to fixed-time decay bounds. Sections 5 and 6 are dedicated to upgrading these inequalities to match Sobolev regularities at the level of $H^2$ and $H^3$, respectively. In section 7, we conclude the analysis by showing that $\Phi$ is regular enough to imply that the continuation criterion for $v$ is valid. The article finishes with an appendix which certifies that the Sobolev regularity of the initial data assumed in every part of the main argument is indeed the right one.


\subsection*{Acknowledgements}

The first author was supported in part by a grant from the Simons Foundation $\#$ 359727.


\section{Preliminaries}


\subsection{Introducing the function $v$ and initial reductions}

We start by writing the equation \eqref{main} in the form
\begin{equation}
\Box_{3+1}u= N(r,u,\nabla u), \label{main-2}
\end{equation}
where
\[
N(r,u,\nabla u):=\frac{-\frac{\sin2u}{r^2}\left(1+u_t^2-u_r^2+\frac{\sin^2u}{r^2}\right)}{1 + \frac{2\sin ^2 u}{r^2}}-\frac{\frac{4\sin^2u}{r^3}u_r}{1 + \frac{2\sin ^2 u}{r^2}}
\]
and
\[
\Box_{3+1}=\p_{tt}-\p_{rr}-\frac{2}{r}\p_r
\]
is the radial wave operator in $\R^{3+1}$. We perform the substitution
\begin{equation}
u(t,r)\,=\,r\,v(t,r)+\varphi(r),
\label{utov}
\end{equation}
with $\varphi:\R_+\to\R_+$ being a smooth, decreasing function, verifying $\varphi\equiv N_1\pi$ on $[0,1]$ and 
$\varphi\equiv 0$ on $[2,\infty)$. We also need to introduce a finer version of $\varphi$, labelled $\varphi_{<1}$, which shares the same smoothness and monotonicity with $\varphi$, but now satisfies $\varphi_{<1}\equiv 1$ on $[0,1/2]$ and $\varphi_{<1}\equiv 0$ on $[1,\infty)$. Furthermore, we write $\varphi_{>1}$ to denote the function $1-\varphi_{<1}$. As a consequence, we obtain
\begin{equation}
\aligned
\Box_{5+1}v=&\frac{1}{r}\Delta_3\varphi+\frac{1}{r}\varphi_{>1}N(r, rv+\varphi, \nabla(rv+\varphi))+\frac{2}{r^2}\varphi_{>1}v\\
&+\varphi_{<1}\left(\frac{1}{r}N(r, rv, \nabla(rv))+\frac{2}{r^2}v\right).
\endaligned
\label{main-v}
\end{equation}
A careful analysis shows that 
\begin{equation}
\aligned
&\frac{1}{r}N(r, rv, \nabla(rv))+\frac{2}{r^2}v\\
&\qquad\qquad=\frac{1}{1+N_0(rv)v^2}\bigg\{N_1(rv)v^3+N_2(rv)v^5+N_3(rv)v(v_t^2-v_r^2)\\
&\qquad\qquad\qquad\qquad\qquad\qquad+N_4(rv)rv^4v_r\bigg\},
\endaligned
\label{Nv-l1}
\end{equation}
with all $N_i=N_i(x)$ being even, analytic, and satisfying
\begin{equation}
\|\partial^kN_i\|_{L^\infty(\R)}\leq C_k, \qquad (\forall)\,k\in\mathbb{N}.
\label{dni-li}
\end{equation}
We refer the reader to \cite{Li} for the precise formulae of these functions.

The previous substitution is motivated by the fact that  it reduces the proof of Theorem \ref{main-th} to the one of the subsequent result concerning \eqref{main-v}, which could be verified through a fairly straightforward argument (e.g., see Subsection 2.3 in Creek \cite{C13}).

\begin{theorem}
Let $(v_0,v_1)$ be radial initial data with
\[
(v_0,v_1)\in H^s\times H^{s-1}(\R^5), \quad s>\frac{7}{2}.
\]
Then there exists a global radial solution $v$ to the Cauchy problem associated to \eqref{main-v} with $(v(0),v_t(0))=(v_0,v_1)$, which satisfies 
\[
v\in C([0,T], H^s(\R^5))\cap C^1([0,T], H^{s-1}(\R^5)), \qquad \forall\,T>0.
\] 
\label{main-th-v}
\end{theorem}
\noindent The approach in proving this theorem relies on a classical result (e.g., see Theorem 6.4.11 in H\"{o}rmander \cite{H97})) that allows us to derive global solutions from local ones, which additionally satisfy a continuation criterion. The entire argument is then reduced to demonstrating the following theorem.

\begin{theorem}
For any $0<T<\infty$ and $s>7/2$, a radial solution $v$ on $[0,T)$ to \eqref{main-v} with $(v(0),v_t(0))\in H^s\times H^{s-1}(\R^5)$  satisfies 
\begin{equation}
\|(1+r)(|v|+|\nabla_{t,x}v|)\|_{L_{t,x}^\infty([0,T)\times \R^5)}\,<\,\infty.
\label{livr}
\end{equation}
\label{main-th-v-2}
\end{theorem}


\subsection{The construction of the auxiliary function $\Phi$ and further reductions}

The proof of Theorem \ref{main-th-v-2} is somewhat indirect, in the sense that we argue for \eqref{livr} by constructing a new function $\Phi$, which satisfies an equation that is easier to study than the one for $v$ (i.e., \eqref{main-v}). The first step in this construction aims to eliminate the derivative terms on the right-hand side of \eqref{main-v} and, for that purpose, we take
\[
\Phi_1(t,r)=\int^{u(t,r)}_{N_1\pi}\left(1+\frac{2\sin^2w}{r^2}\right)^{1/2}dw.
\]
The wave equation satisfied by $\Phi_1$ is given by
\[
\Box_{3+1}\Phi_1=-\frac{2}{r^2}\Phi_1\,+\,\frac{1}{2}\int^{u}_{N_1\pi}\left\{3(A^{3/2}-A^{1/2})+A^{-1/2}-A^{-3/2}\right\}dw,
\] 
with
\[
A=A(r,w):=1+\frac{2\sin^2w}{r^2}.
\]

To deal with the $1/r^2$ singularity, we introduce next
\[
\Phi_2(t,r)=\frac{\Phi_1(t,r)}{r},
\]
which solves
\[
\Box_{5+1}\Phi_2=-\frac{3}{2}\Phi_2\,+\,\frac{1}{2r}\int^{u}_{N_1\pi}\left\{3A^{3/2}+A^{-1/2}-A^{-3/2}\right\}dw.
\]
It seems that we took care of the singularity only to introduce a new one in front of the integral. However, this can be seen to be removable by writing
\[
\frac{1}{2r}=\frac{\varphi_{<1}}{2r}+\frac{\varphi_{>1}}{2r} 
\]
and then making the  change of variable 
\[
w=N_1\pi+ry
\] 
in the integral multiplied by $\varphi_{<1}$. Even if the equation is fixed, a formal argument shows that one might have
\[
\|\Phi_2\|_{L^2(\{r\geq 1\})}=\infty,
\]
which does not fit the expected results of our approach.

To address this final issue, we take
\[
\Phi=\Phi_2\,+\,\frac{1}{3r}\varphi_{>1}\int_{0}^{N_1\pi}\left\{3A^{3/2}+A^{-1/2}-A^{-3/2}\right\}dw.
\]
After careful computations, we deduce that
\begin{equation}
\Phi\,=\,\int_{0}^{v}\left(1+\frac{2\sin^2(ry+\varphi)}{r^2}\right)^{1/2}dy\,+\,\frac{\varphi_{\geq 1/2}}{r^3},
\label{Phi-final}
\end{equation}
with the associated wave equation being
\begin{equation}
\Box_{5+1}\Phi=-\frac{3}{2}\Phi\,+\,\frac{1}{2}\int_{0}^{v}\left\{3\tA^{3/2}+\tA^{-1/2}-\tA^{-3/2}\right\}dy\,+\,\frac{\varphi_{\geq 1/2}}{r^3}.
\label{Box-Phi}
\end{equation}
Above, 
\begin{equation}
\tA=\tA(r,y):=1+\frac{2\sin^2(ry+\varphi(r))}{r^2}
\label{ta-formula}
\end{equation}
and $\varphi_{\geq 1/2}=\varphi_{\geq 1/2}(r)$ is a generic smooth function, with bounded derivatives of all orders and supported in the domain $\{r\geq 1/2\}$, which may change from line to line.

In order to prove Theorem \ref{main-th-v-2}, it is obvious from its formulation that we can additionally assume that $s$ is sufficiently close to $7/2$. In fact, we show that if $7/2<s<18/5$, then
\begin{equation}
\aligned
\|\Phi\|_{L^\infty H^{s}([0,T)\times \R^5)}&+\|\Phi_t\|_{L^\infty H^{s-1}([0,T)\times \R^5)}+\|\Phi_{tt}\|_{L^\infty H^{s-2}([0,T)\times \R^5)}\\
&+\|\Phi_{ttt}\|_{L^\infty L^2([0,T)\times \R^5)}\,<\,\infty,
\endaligned
\label{Phi-li}
\end{equation}
which, coupled with Sobolev embeddings and radial Sobolev inequalities, implies \eqref{livr}.


\section{Notations and analytic toolbox}


\subsection{Notational conventions}

First, we write $A\lesssim B$ to denote $A\leq CB$, where $C$ is a constant depending only upon parameters which are considered fixed throughout the paper. Two such important parameters are the conserved energy \eqref{tote}, written in terms of the initial data $(u_0, u_1)$ in Theorem \ref{main-th} as
\begin{equation}
E:=\int_0^\infty \left\{\left(1 + \frac{2\sin ^2 u_0}{r^2}\right)\frac{u_1^2+u_{0,r}^2}{2}+
\frac{\sin ^2 u_0}{r^2} +\frac{\sin^4 u_0}{2r^4}
\right\}r^2 dr,
\label{tote-0}
\end{equation}
and the time $0<T<\infty$ appearing in Theorem \ref{main-th-v-2}. Moreover, we write $A\sim B$ for the case when both $A\lesssim B$ and $B\lesssim A$ are valid.

Secondly, as is the custom for $w=w(t,x)$, we work with $\nabla w=(\partial_t w, \nabla_x w)$ and
\begin{equation*}
\|w\|_{L^pX(I\times \R^n)}=\|w\|_{L_t^pX_x(I\times \R^n)}= \left(\int_I\|w(t,\cdot)\|_{X(\R^n)}^{p}dt\right)^{1/p},
\end{equation*}
where $X(\R^n)$ is a normed/semi-normed space (e.g., $X=L^q$ or $H^\sigma$ or $\dot{H}^\sigma$) and $I\subseteq \R$ is an arbitrary time interval. Furthermore, for ease of notation, in the case when $I\times \R^n=[0,T)\times \R^5$, we drop from the previous notation the dependence on the domain and simply write
\begin{equation*}
\|w\|_{L^pX}=\|w\|_{L^pX([0,T)\times \R^5)}.
\end{equation*}
This is because the majority of the norms we are dealing with from here on out refers to this particular situation.


\subsection{Analytic toolbox}

Here, we list a number of analytic facts that we will use throughout the argument. First, we recall the classical and general Sobolev embeddings
\begin{align}
H^\sigma(\R^n)\subset L^\infty(\R^n),\qquad \sigma>\frac{n}{2}&, \label{Sob-classic}\\
\dot{H}^{\sigma,p}(\R^n)\subset L^q(\R^n),\qquad 1<p\leq q<\infty, \quad &\sigma=n\left(\frac{1}{p}-\frac{1}{q}\right),  \label{Sob-gen}
\end{align}
and the radial Sobolev estimates (\cite{St77}, \cite{CO09})
\begin{align}
r^{n/2-\sigma}|&f(r)|\lesssim \|f\|_{\dot{H}^\sigma(\R^n)},\qquad \frac{1}{2}<\sigma<\frac{n}{2}, \label{rad-Sob-1}\\
&r^{(n-1)/2}|f(r)|\lesssim \|f\|_{H^1(\R^n)},  \label{rad-Sob-2}
\end{align}
which are valid for radial functions defined on $\R^n$. Related to these, we write down Hardy's inequality (\cite{OK90})
\begin{equation}
\left\|\frac{g}{|x|}\right\|_{L^p(\R^n)}\lesssim \left\|\nabla_x g\right\|_{L^p(\R^n)}, \qquad 1<p<n, \label{Hardy}
\end{equation}
and the interpolation bound (\cite{MR0482275})
\begin{equation}\aligned
&\begin{cases}
\sigma_1\neq \sigma_2, \qquad 1\leq p,p_1,p_2\leq \infty, \qquad 0\leq \theta \leq 1,\\ 
\sigma=(1-\theta)\sigma_1+\theta\sigma_2,\qquad \frac{1}{p}=\frac{1-\theta}{p_1}+\frac{\theta}{p_2},
\end{cases}\\
&\qquad\|g\|_{\dot{H}^{\sigma,p}(\R^n)}\lesssim \|g\|^{1-\theta}_{\dot{H}^{\sigma_1,p_1}(\R^n)}\|g\|^\theta_{\dot{H}^{\sigma_2,p_2}(\R^n)},
\endaligned
\label{interpol-bd}
\end{equation}
both of which hold true for general functions on $\R^n$.

Next, using the Riesz potential $D^\sigma=(-\Delta)^{\sigma/2}$, we record the fractional Leibniz estimate (\cite{GO-14}, \cite{BLi-14})
\begin{equation}
\aligned
&\sigma>0, \quad 1\leq p\leq \infty, \quad 1< p_1, p_2, q_1, q_2\leq \infty, \quad\frac{1}{p}=\frac{1}{p_1}+\frac{1}{q_1}=\frac{1}{p_2}+\frac{1}{q_2}, \\
&\,\|D^\sigma(fg)\|_{L^p(\R^n)}\lesssim \|D^\sigma f\|_{L^{p_1}(\R^n)}\|g\|_{L^{q_1}(\R^n)}+\|f\|_{L^{p_2}(\R^n)}\|D^\sigma g\|_{L^{q_2}(\R^n)}, 
\endaligned
\label{Lbnz-0}
\end{equation}
and the Kato-Ponce type inequalities (\cite{Li-16})
\begin{equation}
\aligned
0<\sigma<2, \quad 1< p, p_1, p_2 &<\infty, \quad \frac{1}{p}=\frac{1}{p_1}+\frac{1}{p_2},\\
\|D^\sigma(fg)-D^\sigma f\, g-f\,D^\sigma g\|_{L^p(\R^n)}&\lesssim \|D^{\sigma/2} f\|_{L^{p_1}(\R^n)}\|D^{\sigma/2}g\|_{L^{p_2}(\R^n)}, 
\endaligned
\label{Lbnz-1}
\end{equation}
\begin{equation}
\aligned
&0<\sigma\leq 1, \qquad 1< p <\infty,\\
\|D^\sigma(fg)-f\,&D^\sigma g\|_{L^p(\R^n)}\lesssim \|D^{\sigma} f\|_{L^{p}(\R^n)}\|g\|_{L^{\infty}(\R^n)}.
\endaligned
\label{Lbnz-2}
\end{equation}
We also mention the well-known Moser bound
\begin{equation}
\|F(f)\|_{H^\sigma(\R^n)}\leq \gamma(\|f\|_{L^{\infty}(\R^n)})\,\|f\|_{H^{\sigma}(\R^n)}, \qquad(\forall)\,f\in L^{\infty}\cap H^\sigma(\R^n; \R^k), 
\label{Moser}
\end{equation}
where $F\in C^\infty(\R^k;\R)$, $F(0)=0$, and $\gamma=\gamma(\sigma)\in C(\R; \R)$. Following this, we recall the Bernstein estimates
\begin{equation}
\aligned
\|P_{>\l}f\|&_{L^p(\R^n)}\lesssim \|f\|_{L^p(\R^n)},\\ \l^\sigma\|P_{>\l}f\|_{L^p(\R^n)}\lesssim& \,\|P_{>\l}D^\sigma f\|_{L^p(\R^n)}, \qquad \sigma\geq 0,
\endaligned
\label{Bernstein}
\end{equation}
where $P_{>\l}$ is a Fourier multiplier localizing the spatial frequencies to the region $\{|\xi|>\l\}$.

Finally, we recount the classical Strichartz inequalities for the $5+1$-dimensional linear wave equation, which take the form
\begin{equation}
\aligned
\|\Psi\|&_{L^pL^q(I\times \R^5)}+\|\Psi\|_{L^\infty\dot{H}^\sigma(I\times \R^5)}+\|\Psi_{t}\|_{L^\infty\dot{H}^{\sigma-1}(I\times \R^5)}\\
&\qquad\lesssim \|\Psi(0)\|_{\dot{H}^\sigma( \R^5)}+\|\Psi_{t}(0)\|_{\dot{H}^{\sigma-1}(\R^5)}+\|\Box\Psi\|_{L^{\bar{p}'}L^{\bar{q}'}(I\times \R^5)},
\endaligned
\label{Str-gen} 
\end{equation}
with $I$ being a time interval and
\[
\begin{cases}
2\leq p\leq \infty, \qquad 2\leq q< \infty, \qquad \frac{1}{p}+\frac{2}{q}\leq 1, \\
1\leq \bar{p}' \leq 2, \qquad\, 1<\bar{q}' \leq 2, \qquad \frac{1}{\bar{p}'}+\frac{2}{\bar{q}'}\geq2, \\
\frac{1}{p}+\frac{5}{q}= \frac{5}{2}-\sigma= -2+\frac{1}{\bar{p}'}+\frac{5}{\bar{q}'}.
\end{cases}
\]
A straightforward consequence of the previous bound is the following generalized energy estimate:
\begin{equation}
\aligned
\|\Psi\|_{L^\infty\dot{H}^\sigma(I\times \R^5)}+\|\Psi_{t}\|_{L^\infty\dot{H}^{\sigma-1}(I\times \R^5)}\lesssim &\|\Psi(0)\|_{\dot{H}^\sigma( \R^5)}+\|\Psi_{t}(0)\|_{\dot{H}^{\sigma-1}(\R^5)}\\
&+\|\Box\Psi\|_{L^1\dot{H}^{\sigma-1}(I\times \R^5)}.
\endaligned
\label{en-hs} 
\end{equation}


\section{Energy-type arguments} \label{sect-en}

In this section, we start in earnest our analysis and prove that
\begin{equation}
\|v\|_{L^\infty H^{1}}+\|v_t\|_{L^\infty L^{2}}\lesssim 1
\label{v-h1}
\end{equation}
and, subsequently,
\begin{equation}
\|\Phi\|_{L^\infty H^{1}}+\|\Phi_t\|_{L^\infty L^{2}}\lesssim 1.
\label{Phi-h1}
\end{equation}
Next, we use this information and the radial Sobolev inequalities \eqref{rad-Sob-1}-\eqref{rad-Sob-2} to derive preliminary decay estimates for both $v$ and $\Phi$, which, in turn, imply valuable asymptotics for $\Phi$ and $\Box\Phi$.

\vspace{.1in}
\hrule


\subsection{Energy-type arguments for $v$} 

Based on the formula \eqref{utov} and Hardy's inequality \eqref{Hardy}, we infer that 
\begin{equation}
\|v_t\|_{L^\infty L^2}\simeq\left\|u_t\right\|_{L^\infty L^2([0,T)\times \R^3)}\lesssim E^{1/2}\lesssim 1
\label{vt-l2}
\end{equation}
and
\begin{equation}
\aligned
\|v_r\|_{L^\infty L^2}&\lesssim\left\|\frac{\varphi_r}{r}\right\|_{L^2(\R^5)}+\left\|u_r\right\|_{L^\infty L^2([0,T)\times \R^3)}+\left\|\frac{u-\varphi}{r^2}\right\|_{L^\infty L^2}\\
&\lesssim 1+E^{1/2}+\left\|\frac{u}{r}\right\|_{L^\infty L^2([0,T)\times \R^3)}\\
&\lesssim 1+E^{1/2}+\left\|u_r\right\|_{L^\infty L^2([0,T)\times \R^3)}\\ 
&\lesssim 1+E^{1/2}\\ &\lesssim 1.
\endaligned
\label{vr-l2}
\end{equation}
Hence, by the fundamental theorem of calculus, we also have
\begin{equation}
\|v\|_{L^\infty L^2}\lesssim\|v(0)\|_{L^2(\R^5)}+T\,E^{1/2}  \lesssim 1,
\label{v-l2}
\end{equation}
which finishes the proof of \eqref{v-h1}.


\subsection{Energy-type arguments for $\Phi$} 

We proceed by using the formula \eqref{Phi-final} to deduce
\begin{equation}
\|\Phi_t\|_{L^\infty L^2}\simeq\left\|\left(1+\frac{2\sin^2 u}{r^2}\right)^{1/2}u_t\right\|_{L^\infty L^2([0,T)\times \R^3)}\lesssim E^{1/2} \lesssim 1.
\label{phitlil2}
\end{equation}
Moreover, a direct argument relying on the same formula yields 
\[\aligned
|\Phi(0)|\lesssim \int_0^{|v(0)|}(1+y)dy+\frac{|\varphi_{\geq 1/2}|}{r^3}\lesssim |v(0)|+|v(0)|^2+\frac{|\varphi_{\geq 1/2}|}{r^3},
\endaligned
\]
which, coupled with the Sobolev embeddings \eqref{Sob-gen}, further implies  
\[
\|\Phi(0)\|_{L^2(\R^5)}\,\lesssim\,1+\|v(0)\|_{L^2(\R^5)}+\|v(0)\|^2_{L^4(\R^5)}\,\lesssim\,1+\|v(0)\|_{H^{5/4}(\R^5)}^2 \lesssim 1.
\]
If we argue like we did for $v$, then we obtain
\begin{equation}
\|\Phi\|_{L^\infty L^2}\lesssim 1+\|v(0)\|^2_{H^{5/4}(\R^5)}+T\,E^{1/2} \lesssim 1.
\label{philil2}
\end{equation}
Thus, in order to conclude the argument for \eqref{Phi-h1}, we need to obtain a favorable estimate for $\|\Phi_r\|_{L^\infty L^2}$, which is quite technical in nature. 

First, we prove the following fixed-time inequality.
\begin{prop}
Under the assumptions of Theorem \ref{main-th-v-2},
\begin{equation}
\int_{\R^5}\left(\Phi_r+2\frac{\Phi}{r}\right)^2 dx\lesssim 1
\label{phir}
\end{equation}
holds true uniformly in time on $[0,T)$.
\end{prop}
\begin{proof}
If we rely again on \eqref{Phi-final}, then we deduce
\[
\Phi_r+2\frac{\Phi}{r}=\left(1+\frac{2\sin^2 u}{r^2}\right)\frac{u_r}{r}+\int_0^v \left(1+\frac{2\sin^2(ry+\varphi)}{r^2}\right)^{-1/2}\frac{1}{r}dy.
\]
Therefore, with the help of \eqref{Hardy} and \eqref{vr-l2}, we infer
\begin{equation}\aligned
\int_{\{r<1/2\}}\left(\Phi_r+2\frac{\Phi}{r}\right)^2 dx&\lesssim \int_{\R^3}\left(1+\frac{2\sin^2 u}{r^2}\right)^2\frac{u^2_r}{r^2} \,dx + \int_{\R^5}\frac{v^2}{r^2} \,dx\\
&\lesssim E+\|v\|^2_{L^\infty\dot{H}^1}\\
&\lesssim 1.
\endaligned
\label{phir-1}
\end{equation}
In the complementary region (i.e., $\{r\geq 1/2\}$), a straightforward analysis using \eqref{Phi-final} yields
\[
|\Phi_r|+\frac{|\Phi|}{r}\lesssim |v_r|+\frac{|v|}{r}+\frac{v^2}{r^2}+\frac{|\tilde{\varphi}_{\geq 1/2}|}{r^3}
\]
where $\tilde{\varphi}_{\geq 1/2}=\tilde{\varphi}_{\geq 1/2}(r)$ is a function sharing the profile of $\varphi_{\geq 1/2}$. Hence, using \eqref{Hardy}, \eqref{rad-Sob-2}, and \eqref{v-h1}, we derive:
\begin{equation*}\aligned
\int_{\{r\geq 1/2\}}\left(\Phi_r+2\frac{\Phi}{r}\right)^2 dx&\lesssim 1+\|v\|^2_{L^\infty \dot{H}^1}+\int_{\{r\geq 1/2\}}\frac{v^4}{r^4}\,dx \\
&\lesssim 1+\|v\|^2_{L^\infty \dot{H}^1}+\|v\|^4_{L^\infty H^1}  \int_{\{r\geq 1/2\}}\frac{1}{r^{12}}\,dx\\
&\lesssim 1+\|v\|^4_{L^\infty H^1}\\
&\lesssim 1.
\endaligned
\end{equation*}
Together with \eqref{phir-1}, this bound implies \eqref{phir}.
\end{proof}

Next, we show that another related fixed-time estimate is true.

\begin{prop}
Under the assumptions of Theorem \ref{main-th-v-2}, 
\begin{equation}
\aligned
\int_{\R^5}\left\{\Phi_r^2-\frac{9}{4}\frac{\Phi^2}{r^2}\right\}dx \lesssim 1.
\endaligned
\label{phir-phi}
\end{equation}
holds true uniformly in time on $[0,T)$.
\end{prop}
\begin{proof}
The first step in the argument is to prove
\begin{equation}
\aligned
\frac{d}{dt}\left\{\int_{\R^5}\frac{\Phi^2_t+\Phi^2_r}{2}\,dx\right\}-\int_{\R^5}\left\{\varphi_{<r_0}\Phi_t\int_0^v\frac{3}{2}\left(\tA^{3/2}-\tA^{1/2}\right)dy\right\}dx 
\lesssim \frac{1}{r^2_0},
\label{phir-phi-v1}
\endaligned
\end{equation}
where $\varphi_{<r_0}=\varphi_{<r_0}(r)=\varphi_{<1}(r/r_0)$ and $r_0\leq 1$ is a scale to be further calibrated. To achieve this, we use \eqref{Phi-final}, \eqref{Box-Phi}, and $\tA\geq 1$ to obtain
\begin{equation}
|v|\leq|\Phi|+\frac{\left|\varphi_{\geq 1/2}\right|}{r^3}
\label{v-bd}
\end{equation}
and
\begin{equation*}
\Box_{5+1}\Phi=\frac{1}{2}\int_{0}^{v}\left\{3(\tA^{3/2}-\tA^{1/2})+(\tA^{-1/2}-\tA^{-3/2})\right\}dy\,+\,\frac{\varphi_{\geq 1/2}}{r^3}.
\end{equation*}
If we multiply this equation by $\Phi_t$ and integrate the outcome on $\R^5$, then we deduce
\begin{equation}
\aligned
\frac{d}{dt}&\left\{\int_{\R^5}\frac{\Phi^2_t+\Phi^2_r}{2}\,dx\right\}-\int_{\R^5}\left\{\varphi_{<r_0}\Phi_t\int_0^v\frac{3}{2}\left(\tA^{3/2}-\tA^{1/2}\right)dy\right\}dx\\ 
&\quad=\int_{\R^5}\left\{(1-\varphi_{<r_0})\Phi_t\int_0^v\frac{3}{2}\left(\tA^{3/2}-\tA^{1/2}\right)dy\right\}dx\\
&\quad\quad+\int_{\R^5}\left\{\Phi_t\left[ \int_0^v\frac{1}{2}\left(\tA^{-1/2}-\tA^{-3/2}\right)dy+\frac{\varphi_{\geq 1/2}}{r^3}\right]\right\}dx.
\endaligned
\label{ddtphi}
\end{equation}
Due to $\tA\geq 1$, \eqref{v-bd}, \eqref{phitlil2}, and \eqref{v-l2}, it follows that
\begin{equation}
\aligned
\int_{\R^5}&\left\{\Phi_t\left[ \int_0^v\frac{1}{2}\left(\tA^{-1/2}-\tA^{-3/2}\right)dy+\frac{\varphi_{\geq 1/2}}{r^3}\right]\right\}dx\\
&\qquad\qquad\qquad\lesssim \int_{\R^5}\left\{|\Phi_t|\left(|v|+\frac{|\varphi_{\geq 1/2}|}{r^3}\right)\right\}dx\\
&\qquad\qquad\qquad\lesssim \|\Phi_t\|_{L^\infty L^2}\left(1+\|v\|_{L^\infty L^2}\right)\\
&\qquad\qquad\qquad\lesssim 1.
\endaligned
\label{ddtphi-2}
\end{equation}
On the other hand, with the help of \eqref{Phi-final}, we derive
\begin{equation*}
\left|(1-\varphi_{<r_0})\int_0^v\left(\tA^{3/2}-\tA^{1/2}\right)dy\right|\lesssim \frac{1}{r_0^2}\int_0^{|v|}\tA^{1/2}dy\lesssim  \frac{1}{r_0^2}\left(|\Phi|+\frac{\left|\varphi_{\geq 1/2}\right|}{r^3}\right).
\end{equation*}
If we also factor in \eqref{phitlil2} and \eqref{philil2}, then this estimate implies
\begin{equation*}
\aligned
\int_{\R^5}\left\{(1-\varphi_{<r_0})\Phi_t\int_0^v\frac{3}{2}\left(\tA^{3/2}-\tA^{1/2}\right)dy\right\}dx\lesssim \frac{\|\Phi_t\|_{L^\infty L^2}}{r^2_0}\left(\|\Phi\|_{L^\infty L^2}+1\right)\lesssim \frac{1}{r^2_0}.
\endaligned
\end{equation*}
It is clear now that \eqref{phir-phi-v1} holds as a result of this bound, \eqref{ddtphi}, and \eqref{ddtphi-2}.

The second step in the proof of \eqref{phir-phi} consists in rewriting the second term on the right-hand side of \eqref{phir-phi-v1} in a friendlier format. For this purpose, we introduce
\[
H(r,w):=\frac{3}{2}\int_0^w \left(1+\frac{2\sin^2(rx)}{r^2}\right)^{1/2}\left\{\int_0^x \left(1+\frac{2\sin^2(ry)}{r^2}\right)^{1/2}\frac{2\sin^2(ry)}{r^2}\,dy\right\}dx.
\]
If we take advantage of \eqref{Phi-final} and $r_0\leq 1$, then we obtain
\[
\aligned
\frac{d}{dt}&\left\{\varphi_{<r_0}H(r,v)\right\}\\
&= \varphi_{<r_0}\left(1+\frac{2\sin^2(rv)}{r^2}\right)^{1/2}v_t\int_0^v \left(1+\frac{2\sin^2(ry)}{r^2}\right)^{1/2}\frac{3\sin^2(ry)}{r^2}\,dy\\
&=\varphi_{<r_0}\Phi_t\int_0^v \left(1+\frac{2\sin^2(ry+\varphi(r))}{r^2}\right)^{1/2}\frac{3\sin^2(ry+\varphi(r))}{r^2}\,dy\\
&= \varphi_{<r_0}\Phi_t\int_0^v\frac{3}{2}\left(\tA^{3/2}-\tA^{1/2}\right)dy.
\endaligned
\] 
Hence, we can restate \eqref{phir-phi-v1} as
\begin{equation*}
\aligned
\frac{d}{dt}\left\{\int_{\R^5}\left\{\frac{\Phi^2_t+\Phi^2_r}{2}-\varphi_{<r_0}\,H(r,v)\right\}dx\right\}\lesssim \frac{1}{r^2_0}.
\endaligned
\end{equation*}

The third step of this argument involves integrating the previous estimate over the interval $[0,t]\subset [0,T)$, which leads to 
\begin{equation}
\aligned
\int_{\R^5}&\left\{\frac{\Phi^2_r}{2}-\varphi_{<r_0}\,H(r,v)\right\}dx\\
&\qquad\lesssim \int_{\R^5}\left\{\frac{\Phi^2_r(0)}{2}-\varphi_{<r_0}\,H(r,v(0))\right\}dx+\|\Phi_t\|^2_{L^\infty L^2}+\frac{T}{r^2_0}\\
&\qquad\lesssim \int_{\R^5}\left\{\frac{\Phi^2_r(0)}{2}-\varphi_{<r_0}\,H(r,v(0))\right\}dx+\frac{1}{r^2_0}.
\endaligned
\label{phir-l2-1}
\end{equation}
We address first the integral term on the right-hand side, for which a direct computation based on \eqref{Phi-final} and \eqref{ta-formula} reveals that
\[
\aligned
\Phi_r(0)\,=\,\left(1+\frac{2\sin^2u(0)}{r^2}\right)^{1/2}v_r(0)+\frac{\varphi_{\geq 1/2}}{r^3}+\frac{1}{2}\int_{0}^{v(0)}\left\{\tA^{-1/2}\tA_r\right\}dy,
\endaligned
\]
with
\begin{equation}
\tA_r=\frac{-4\sin^2(ry+\varphi)}{r^3}+\frac{2\sin2(ry+\varphi)\cdot(y+\varphi_r)}{r^2}.
\label{ta-r}
\end{equation}
Using the formula \eqref{utov}, it is easy to derive
\begin{equation*}
\frac{\sin^2u(0)}{r^2}\lesssim 1+v^2(0).
\end{equation*}
Moreover, a direct analysis using Maclaurin series shows that, when $r\geq 1$, we have
\begin{equation}
\left|\frac{-2\sin^2(ry+\varphi)}{r^3}+\frac{\sin2(ry+\varphi)\cdot(y+\varphi_r)}{r^2}\right|\lesssim\frac{1+|y|}{r^2},
\label{ta-r-g1}
\end{equation}
while for $r<1$, we get
\begin{equation}
\left|\frac{-2\sin^2(ry+\varphi)}{r^3}+\frac{\sin2(ry+\varphi)\cdot(y+\varphi_r)}{r^2}\right|\lesssim ry^4.
\label{ta-r-l1}
\end{equation}
Hence, by collecting the last five mathematical statements and applying the Sobolev embeddings \eqref{Sob-gen}, we conclude that
\begin{equation}
\aligned
\int_{\R^5}\Phi^2_r(0)\,dx\lesssim& \int_{\R^5}\left\{(1+v^2(0))v^2_r(0)+\frac{\varphi^2_{\geq 1/2}}{r^6}\right\}dx\\
&+\int_{\{r\geq 1\}}\left\{\frac{v^2(0)+v^4(0)}{r^4}\right\}dx+\int_{\{r< 1\}}\left\{r^2v^{10}(0)\right\}dx.\\
\lesssim& \ 1+\|v(0)\|^2_{\dot{H}^{1}(\R^5)}+\|v(0)\|^4_{H^{7/4}(\R^5)}+\|v(0)\|^2_{L^2(\R^5)}\\
&\ +\|v(0)\|^4_{H^{5/4}(\R^5)}+\|v(0)\|^{10}_{H^{2}(\R^5)}
\\
\lesssim& \ 1.
\endaligned
\label{phir-l2}
\end{equation}
Using the basic estimate $|\sin x|\leq|x|$, we easily infer that
\[
|H(r,w)|\lesssim w^4+w^6
\]
and, as a consequence, we have
\begin{equation*}
\aligned
\left|\int_{\R^5}\varphi_{<r_0}\,H(r,v(0))dx\right|&\lesssim \|v(0)\|^4_{L^4(\R^5)}+\|v(0)\|^6_{L^6(\R^5)}\\&\lesssim \|v(0)\|^4_{H^{5/4}(\R^5)}+\|v(0)\|^6_{H^{5/3}(\R^5)}
\\&\lesssim 1.
\endaligned
\end{equation*}
Together with \eqref{phir-l2-1} and \eqref{phir-l2}, this bound yields
\begin{equation}
\aligned
\int_{\R^5}\left\{\frac{\Phi^2_r}{2}-\varphi_{<r_0}\,H(r,v)\right\}dx\lesssim \frac{1}{r^2_0}.
\endaligned
\label{phir-l2-2}
\end{equation}

In the last step of this proof, we show that if we choose $r_0$ to be sufficiently small, then the following estimate holds:
\begin{equation}
|\varphi_{<r_0}\,H(r,v)|\leq \frac{9}{8}\frac{\Phi^2}{r^2}.
\label{phi98}
\end{equation}
It is clear that, jointly with \eqref{phir-l2-2}, this inequality implies \eqref{phir-phi}. We claim that a calculus-level analysis finds that
\[
|H(r,w)|\leq \frac{9}{8r^2} \left(\int^w_0\left(1+\frac{2\sin^2(rx)}{r^2}\right)^{1/2}dx\right)^2, \qquad \forall\, (r,w)\in(0,r_1)\times\R,
\]
is true if $r_1$ is sufficiently small. Therefore, by choosing 
\[
r_0<\min\left\{\frac{1}{2}, r_1\right\}
\]
and also using \eqref{Phi-final}, we infer
\[
|\varphi_{<r_0}\,H(r,v)|\leq \frac{9\,\varphi_{<r_0}}{8r^2} \left(\int^v_0\left(1+\frac{2\sin^2(rx)}{r^2}\right)^{1/2}dx\right)^2=\frac{9\,\varphi_{<r_0}}{8r^2}\,\Phi^2\leq \frac{9}{8}\frac{\Phi^2}{r^2},
\]
which finishes the argument for \eqref{phi98}.
\end{proof}

We can finally now invoke the basic inequality
\begin{equation*}
\frac{1}{10}b^2\leq \frac{3}{2}\left(b+2a\right)^2+b^2-\frac{9}{4}a^2,
\end{equation*}
which, coupled to \eqref{phir} and \eqref{phir-phi}, yields the desired control over $\|\Phi_r\|_{L^\infty L^2}$ in the form of
\begin{equation}
\aligned
\|\Phi_r\|_{L^\infty L^2}\lesssim 1.
\endaligned
\label{phirlil2}
\end{equation}


\subsection{Preliminary decay estimates and asymptotics}

First, we take advantage of \eqref{v-h1} and \eqref{Phi-h1} and derive decay estimates for both $\Phi$ and $v$. 
\begin{prop}
Under the assumptions of Theorem \ref{main-th-v-2}, we have
\begin{align}
|\Phi(t,r)|&\lesssim \min\left\{\frac{1}{r^2}, \frac{1}{r^{3/2}}\right\},\label{decay-Phi}\\
|v(t,r)|&\lesssim \min\left\{\frac{1}{r^2}, \frac{1}{r^{3/4}}\right\}.\label{decay-v}
\end{align}
\end{prop}
\begin{proof}
The radial Sobolev inequalities \eqref{rad-Sob-1} and \eqref{rad-Sob-2} easily imply
\begin{equation*}
\aligned
|\Phi(t,r)|&\lesssim \min\left\{\frac{1}{r^2}, \frac{1}{r^{3/2}}\right\}\|\Phi\|_{L^\infty H^1},\\
|v(t,r)|&\lesssim \min\left\{\frac{1}{r^2}, \frac{1}{r^{3/2}}\right\}\|v\|_{L^\infty H^1},
\endaligned
\end{equation*}
thus proving \eqref{decay-Phi} and half of \eqref{decay-v} on the basis of \eqref{v-h1} and \eqref{Phi-h1}. 

For the other half of \eqref{decay-v}, we work in the regime when $r\ll 1$. If we choose $r<1/2$ and use \eqref{Phi-final}, then we obtain
\[
\Phi\,=\,\int_{0}^{v}\left(1+\frac{2\sin^2(ry)}{r^2}\right)^{1/2}dy
\]
and, consequently,
\[
|\Phi|\,\geq\,\int_{0}^{|v|}\frac{|\sin(ry)|}{r}dy=\frac{1}{r^2}\int_{0}^{r|v|}|\sin z|dz.
\]
Relying on \eqref{decay-Phi}, we deduce
\[
\int_{0}^{r|v|}|\sin z|dz\lesssim r^{1/2},
\]
which shows that $r|v|\leq \pi/2$ if $r$ is sufficiently small. Hence, we can argue that in this scenario we have
\[
\frac{r^2v^2}{2}=\int_{0}^{r|v|}|z|dz\lesssim \int_{0}^{r|v|}|\sin z|dz\lesssim r^{1/2},
\] 
thus proving the desired bound.
\end{proof}

Next, we use these decay estimates and obtain asymptotics for $\Phi$ and $\Box\Phi$ in terms of $v$, which turn out to be very important in further arguments.
\begin{prop}
Under the assumptions of Theorem \ref{main-th-v-2}, we have
\begin{equation}
|\Phi|\sim |v|+v^2 \quad \text{and} \quad|\Box\Phi|\sim |v|^3+v^4\qquad\text{if}\quad r\ll 1
\label{rll1}
\end{equation}
and
\begin{equation}
\left|\Phi-\frac{\varphi_{\geq 1/2}}{r^3}\right|\sim |v| \quad \text{and} \quad\left|\Box\Phi-\frac{\varphi_{\geq 1/2}}{r^3}\right|\lesssim \frac{|v|}{r^2}\qquad\text{if}\quad r\gtrsim 1.
\label{rgrt1}
\end{equation}
\end{prop}
\begin{proof}
We start by rewriting \eqref{Phi-final} and \eqref{Box-Phi} in the form
\[
\aligned
\Phi-\frac{\varphi_{\geq 1/2}}{r^3}&=\int_0^v \tilde{A}^{1/2}\,dy,\\
\Box\Phi-\frac{\varphi_{\geq 1/2}}{r^3}&=\int_0^v \left(\frac{3}{2}\tilde{A}^{1/2}+\frac{1}{2}\tilde{A}^{-3/2}\right)(\tilde{A}-1)\,dy.
\endaligned
\]
If we choose $r<1/2$, then
\[
\tilde{A}=1+\frac{2\sin^2(ry)}{r^2}, \qquad \varphi_{\geq 1/2}(r)=0. 
\]
Moreover, due to \eqref{decay-v}, it follows that, by further calibrating $r$ to be sufficiently small, we can guarantee that $r|v|\leq 1$. Therefore, we derive
\[
|\Phi|=\int_0^{|v|}\left(1+\frac{2\sin^2(ry)}{r^2}\right)^{1/2}dy\sim \int_0^{|v|}\left(1+y\right)dy\sim |v|+v^2
\] 
and
\[
\aligned
|\Box\Phi|&=\int_0^{|v|}\left\{\frac{3}{2}\left(1+\frac{2\sin^2(ry)}{r^2}\right)^{1/2}+\frac{1}{2}\left(1+\frac{2\sin^2(ry)}{r^2}\right)^{-3/2}\right\}\frac{2\sin^2(ry)}{r^2}\,dy\\
&\sim \int_0^{|v|}\left(1+y\right)y^2\,dy\sim |v|^3+v^4,
\endaligned
\]
which proves \eqref{rll1}.

In the complementary case when $r\gtrsim 1$, we easily have
\[
1\leq \tilde{A}\leq 1+\frac{2}{r^2}\sim 1
\] 
and the derivation of \eqref{rgrt1} follows exactly like above. \end{proof}


\section{$H^2$-type analysis} 

In this section, we take the next step in our analysis and show that
\begin{equation}
\|\Phi\|_{L^\infty \dot{H}^2}+\|\Phi_t\|_{L^\infty \dot{H}^1}+\|\Phi_{tt}\|_{L^\infty L^2}\lesssim 1.
\label{H2-norm}
\end{equation}
This is done by first deriving a wave equation for $\Phi_t$, which is then investigated by using the Strichartz estimates \eqref{Str-gen}. As a result, we obtain the desired Sobolev regularity for  both $\Phi_t$ and $\Phi_{tt}$. Jointly  with the main equation satisfied by $\Phi$ (i.e., \eqref{Box-Phi}), this information yields that $\Phi\in L^\infty \dot{H}^2$. Following this, we improve the decay estimates \eqref{decay-Phi} and \eqref{decay-v}.

\vspace{.1in}
\hrule


\subsection{Argument for the $\dot{H}^1$ and $L^2$ regularities of  $\Phi_{t}$ and $\Phi_{tt}$}

We start by noticing that a simple differentiation with respect to $t$ of both \eqref{Phi-final} and \eqref{Box-Phi} produces
\begin{equation}
\Phi_t=\tA^{1/2}(r,v)v_t=\left(1+\frac{2\sin^2u}{r^2}\right)^{1/2}v_t
\label{Phit-final}
\end{equation}
and
\begin{equation}
\aligned
\Box_{5+1}\Phi_t&=-\frac{3}{2}\Phi_t+\frac{1}{2}\left(3\tA^{3/2}(r,v)+\tA^{-1/2}(r,v)-\tA^{-3/2}(r,v)\right)v_t\\
&= \left(\tA(r,v)-1\right)\left(\frac{3}{2}+\frac{\tA^{-2}(r,v)}{2}\right)\Phi_t.
\endaligned
\label{Box-Phit}
\end{equation}

From these equations, it is clear that the expression $\tA(r,v)$ will play an important role moving forward. For this purpose, we rely on the decay estimate \eqref{decay-v} to easily infer
\begin{equation}
|\sin u|\lesssim \min\left\{\frac{1}{r}, r^{1/4}\right\},
\label{sinu-r}
\end{equation}
which leads to
\begin{equation}
\tA(r,v)-1=|\tA(r,v)-1|\lesssim \min\left\{\frac{1}{r^4}, \frac{1}{r^{3/2}}\right\}
\label{decay-A}
\end{equation}
and, subsequently, 
\begin{equation}
\frac{3}{2}+\frac{\tA^{-2}(r,v)}{2}\sim 1.
\label{ta-2}
\end{equation}
Moreover, the conservation of energy \eqref{tote} and \eqref{sinu-r} imply
\begin{equation}
\aligned
\|\tA(r,v)-1\|^{10/3}_{L^{10/3}(\R^5)}&\sim \int_0^\infty\frac{|\sin u|^{20/3}}{r^{8/3}}\,dr\\
&\lesssim \int_0^1\frac{\sin^4u}{r^2}\,dr+\int_1^\infty\frac{1}{r^{28/3}}\,dr\\
&\lesssim E+1\\&\lesssim 1.
\endaligned
\label{ta-103}
\end{equation}
We now have all that is needed to prove that $\Phi_t$ and $\Phi_{tt}$ have $\dot{H}^1$ and $L^2$ regularities, respectively.

\begin{prop}
Under the assumptions of Theorem \ref{main-th-v-2},
\begin{equation}
\|\Phi_t\|_{L^p L^q}+\|\Phi_{t}\|_{L^\infty \dot{H}^{1}}+\|\Phi_{tt}\|_{L^\infty L^2}\lesssim 1
\label{Phit-h1}
\end{equation}
holds true for all pairs $(p,q)$ satisfying 
\begin{equation}
2\leq p\leq \infty, \qquad 2\leq q<\infty, \qquad \frac{1}{p}+\frac{2}{q}\leq 1, \qquad \frac{1}{p}+\frac{5}{q}=\frac{3}{2}.
\label{pq-h2}
\end{equation}
\label{prop-Phit-h1}
\end{prop}
\begin{proof}
We start by applying the Strichartz estimates \eqref{Str-gen} to the nonlinear wave equation \eqref{Box-Phit} for the case when $\sigma=1$ and $(\bar{p}',\bar{q}')=(1,2)$. We derive that
\begin{equation}
\aligned
\|\Phi_t\|&_{L^pL^q(I\times \R^5)}+\|\Phi_t\|_{L^\infty\dot{H}^1(I\times \R^5)}+\|\Phi_{tt}\|_{L^\infty L^2(I\times \R^5)}\\
&\qquad\lesssim \|\Phi_t(a)\|_{\dot{H}^1( \R^5)}+\|\Phi_{tt}(a)\|_{L^2(\R^5)}+\|\Box\Phi_t\|_{L^1L^2(I\times \R^5)}
\endaligned
\label{Phit-h1-str}
\end{equation}
is valid for all intervals $I=[a,b]\subset [0,T)$ and pairs $(p,q)$ satisfying \eqref{pq-h2}. For the last term on the right-hand side, we use \eqref{ta-2} and \eqref{ta-103} to obtain
\begin{equation}
\aligned
\|\Box\Phi_t\|_{L^1L^2(I\times \R^5)}&\lesssim  \|\tA(r,v)-1\|_{L^2L^{10/3}(I\times \R^5)}\|\Phi_t\|_{L^2L^5(I\times \R^5)}\\
&\lesssim  |I|^{1/2}\|\Phi_t\|_{L^2L^5(I\times \R^5)}.
\endaligned
\label{Box-Phit-l1l2}
\end{equation}

It is easy to see that we are allowed to have $(p,q)=(2,5)$ in \eqref{Phit-h1-str} and, as a result, we deduce  
\begin{equation*}
\aligned
\|\Phi_t\|&_{L^2L^5(I\times \R^5)}+\|\Phi_t\|_{L^\infty\dot{H}^1(I\times \R^5)}+\|\Phi_{tt}\|_{L^\infty L^2(I\times \R^5)}\\
&\qquad\lesssim \|\Phi_t(a)\|_{\dot{H}^1( \R^5)}+\|\Phi_{tt}(a)\|_{L^2(\R^5)}+|I|^{1/2}\|\Phi_t\|_{L^2L^5(I\times \R^5)}.
\endaligned
\end{equation*}
Recalling our notational conventions, it follows that  for $|I|\sim 1$, yet sufficiently small, we have
\[
M(I)\lesssim \|\Phi_t(a)\|_{\dot{H}^{1}( \R^5)}+\|\Phi_{tt}(a)\|_{L^2(\R^5)},
\]
where
\[
M(I):=\|\Phi_t\|_{L^2L^{5}(I\times \R^5)}+\|\Phi_t\|_{L^\infty\dot{H}^{1}(I\times \R^5)}+\|\Phi_{tt}\|_{L^\infty L^2(I \times \R^5)}.
\]
Therefore, if one chooses $T_1$ to be the maximal length of an interval for which the previous bound holds true, then
\[
\aligned
M([(k+1)T_1, (k+2)T_1])&\lesssim \|\Phi_t((k+1)T_1)\|_{\dot{H}^{1}( \R^5)}+\|\Phi_{tt}((k+1)T_1)\|_{L^2(\R^5)}\\&\lesssim M([kT_1, (k+1)T_1])
\endaligned
\]
holds true for as long as 
\[
0\leq kT_1<(k+2)T_1<T,
\] 
with $k$ being a nonnegative integer. Given that  the assumptions of Theorem \ref{main-th-v-2} guarantee, through the Appendix, that
\[
M([0,T_1])\lesssim \|\Phi_t(0)\|_{\dot{H}^{1}( \R^5)}+\|\Phi_{tt}(0)\|_{L^2(\R^5)}\lesssim 1,
\]
it is clear that the previous facts lead to 
\[\|\Phi_t\|_{L^2L^{5}}+\|\Phi_t\|_{L^\infty\dot{H}^1}+\|\Phi_{tt}\|_{L^\infty L^2}\lesssim 1.
\]

Now, we can go back to \eqref{Box-Phit-l1l2} and claim 
\[
\|\Box\Phi_t\|_{L^1L^2}\lesssim 1.
\] 
Coupled to \eqref{Phit-h1-str}, this estimate implies that
\[
\|\Phi_t\|_{L^pL^q}\lesssim 1
\]
also holds true for $(p,q)\neq (2,5)$ satisfying \eqref{pq-h2} and, thus, finishes the proof.
\end{proof}


\subsection{$\dot{H}^2$ regularity for $\Phi$ and improved decay estimates}

Using the newfound regularities for $\Phi_t$ and $\Phi_{tt}$ in conjunction with the wave equation \eqref{Box-Phi}, we show that, at this stage of the analysis, $\Phi$ has $\dot{H}^2$ regularity. 

\begin{prop}
Under the assumptions of Theorem \ref{main-th-v-2}, it is true that
\begin{equation}
\|\Phi\|_{L^\infty\dot{H}^2}\lesssim 1. 
\label{Phi-lih2}
\end{equation}
\end{prop}
\begin{proof}
For $0\leq \sigma\leq 1$, we perform the decomposition
\begin{equation}
D^{\sigma-1}\Delta \Phi=(1-P_{>1})D^{\sigma-1}\Delta \Phi+P_{>1}D^{\sigma-1}\Delta \Phi,
\label{dphi-decomp}
\end{equation}
where $P_{>1}$ is a Fourier multiplier localizing the spatial frequencies to the region $\{|\xi|\geq 2\}$. For the first term on the right-hand side, we can easily infer based on \eqref{Phi-h1} that
\begin{equation}
\|(1-P_{>1})D^{\sigma-1}\Delta \Phi\|_{L^\infty L^2}\lesssim \|\Phi\|_{L^\infty \dot{H}^1}\lesssim 1.
\label{delta-p-1}
\end{equation}
For the second term, we rely on \eqref{Phit-h1}, the Bernstein's inequalities \eqref{Bernstein}, and the Sobolev embeddings \eqref{Sob-gen} to derive that 
\begin{equation}
\aligned
\|P_{>1}D^{\sigma-1}\Delta \Phi\|_{L^\infty L^2}
&\lesssim \|P_{>1}D^{\sigma-1} \Phi_{tt}\|_{L^\infty L^2}+\|P_{>1}D^{\sigma-1}\Box \Phi\|_{L^\infty L^2}\\
&\lesssim \|\Phi_{tt}\|_{L^\infty L^2}+\|P_{>1}D^{\sigma-1}\left(\varphi_{>1}\Box \Phi\right)\|_{L^\infty L^2}\\
&\qquad\qquad+\|P_{>1}D^{\sigma-1}\left(\varphi_{<1}\Box \Phi\right)\|_{L^\infty L^2}\\
&\lesssim 1+\|\varphi_{>1}\Box \Phi\|_{L^\infty L^2}+\|\varphi_{<1}\Box \Phi\|_{L^\infty L^{p(\sigma)}},
\endaligned
\label{delta-p-2}
\end{equation}
where
\begin{equation}
\frac{1}{p(\sigma)}=\frac{1}{2}+\frac{1-\sigma}{5}.
\label{ps}
\end{equation}
By employing \eqref{rgrt1} and \eqref{decay-v}, we now deduce that
\begin{equation}
\|\varphi_{>1}\Box \Phi\|_{L^\infty L^2}\lesssim \left\|\frac{\varphi_{\geq 1/2}}{r^3}\right\|_{L^\infty L^2}+\left\| \varphi_{>1}\frac{|v|}{r^2}\right\|_{L^\infty L^2}\lesssim 1.
\label{delta-p-3}
\end{equation}
For the norm involving $\varphi_{<1}\Box \Phi$, if we use \eqref{rll1}, we obtain, on the account of the Sobolev embeddings \eqref{Sob-gen} and \eqref{Phi-h1}, that
\begin{equation}
\|\varphi_{<1}\Box \Phi\|_{L^\infty L^{p(\sigma)}}\lesssim \|\varphi_{<1} \Phi^2\|_{L^\infty L^{p(\sigma)}}\lesssim \|\Phi\|^2_{L^\infty L^{2p(\sigma)}}\lesssim \|\Phi\|^2_{L^\infty H^1}\lesssim 1
\label{delta-p-4}
\end{equation}
holds true if $2\leq 2p(\sigma)\leq 10/3$. Due to \eqref{ps}, it is easy to see that this happens when $0\leq \sigma \leq 1/2$ and, consequently,
\[
\|\Phi\|_{L^\infty \dot{H}^{3/2}}\lesssim 1.
\] 
Therefore, based on the Sobolev embeddings \eqref{Sob-gen}, \eqref{Phi-h1}, and \eqref{dphi-decomp}-\eqref{delta-p-4}, we infer that
\begin{equation*}
\|\Phi\|_{L^\infty L^4}\lesssim\|\Phi\|_{L^\infty H^{5/4}}\lesssim 1,
\end{equation*} 
It follows that we can upgrade \eqref{delta-p-4} to read
\[
\|\varphi_{<1}\Box \Phi\|_{L^\infty L^2} \lesssim \|\Phi\|^2_{L^\infty L^{4}}\lesssim 1,
\]
which, coupled to \eqref{dphi-decomp}-\eqref{delta-p-3}, allows us to derive that
\[
\|\Delta \Phi\|_{L^\infty L^2}\lesssim 1,
\]
thus proving \eqref{Phi-lih2}.
\end{proof}

Next,  a direct application of the radial Sobolev inequalities \eqref{rad-Sob-1} and \eqref{rad-Sob-2} and of the asymptotic equation \eqref{rll1}, in the context of the $\dot{H}^2$ regularity for $\Phi$, leads to the following upgrade for the previous decay estimates satisfied by $\Phi$, $v$, and $\tA(r,v)-1$. 

\begin{prop}
Under the assumptions of Theorem \ref{main-th-v-2}, we have
\begin{align}
|\Phi(t,r)|&\lesssim \min\left\{\frac{1}{r^2}, \frac{1}{r^{1/2}}\right\},\label{decay-Phi-2}\\
|v(t,r)|&\lesssim \min\left\{\frac{1}{r^2}, \frac{1}{r^{1/4}}\right\},\label{decay-v-2}\\
|\tA(r,v)-1|&\lesssim \min\left\{\frac{1}{r^4}, \frac{1}{r^{1/2}}\right\}.\label{decay-A-2}
\end{align}
\end{prop}


\section{$H^3$-type analysis} 

Here, we show that
\begin{equation}
\|\Phi\|_{L^\infty \dot{H}^3}+\|\Phi_t\|_{L^\infty \dot{H}^2}+\|\Phi_{tt}\|_{L^\infty \dot{H}^1}+\|\Phi_{ttt}\|_{L^\infty L^2}\lesssim 1.
\label{H3-norm}
\end{equation}
The approach is similar to the one used in the previous section, in the sense that we start by writing a wave equation for $\Phi_{tt}$ and we analyze it through the Strichartz estimates \eqref{Str-gen}. This yields the expected regularity for  both $\Phi_{tt}$ and $\Phi_{ttt}$. Next, we tie these regularities to equations satisfied by $\Phi_t$ and $\Phi_r$ to deduce that $\Phi_t\in L^\infty \dot{H}^2$ and $\Phi\in L^\infty \dot{H}^3$, respectively. Finally, we continue to improve the decay rates for  $\Phi$, $v$, and $\tA(r,v)-1$.

\vspace{.1in}
\hrule

\subsection{Derivation of $\dot{H}^1$ and $L^2$ regularities for  $\Phi_{tt}$ and $\Phi_{ttt}$}

By differentiating \eqref{Phit-final} and \eqref{Box-Phit} with respect to $t$, we obtain
\begin{equation}
\Phi_{tt}=\tA^{1/2}(r,v)v_{tt}+ \tA^{-1/2}(r,v)\,\frac{\sin(2u)}{r}\,v^2_{t}
\label{Phitt-final}
\end{equation}
and
\begin{equation}
\aligned
\Box_{5+1}\Phi_{tt}=& \left(\tA(r,v)-1\right)\left(\frac{3}{2}+\frac{\tA^{-2}(r,v)}{2}\right)\Phi_{tt}\\
&+\left(\frac{3}{2}-\frac{\tA^{-2}(r,v)}{2}+\tA^{-3}(r,v)\right)\partial_t(\tA(r,v))\Phi_{t}.
\endaligned
\label{Box-Phitt}
\end{equation}
It is important to notice that \eqref{Phit-final} and \eqref{decay-A} imply
\begin{equation}
|\partial_t(\tA(r,v))|=2\frac{|\sin(2u)|}{r}\,|v_t|\lesssim \tA^{1/2}(r,v)\,|v_t|=|\Phi_t|
\label{tA-Phit}
\end{equation}
and 
\[
\frac{3}{2}-\frac{\tA^{-2}(r,v)}{2}+\tA^{-3}(r,v)\sim 1,
\]
which lead to 
\begin{equation}
|\Box\Phi_{tt}|\lesssim\Phi_{t}^2+ |(\tA(r,v)-1)\Phi_{tt}|.
\label{bd-box-phitt}
\end{equation}
These are all the necessary prerequisites to argue for the desired regularities for $\Phi_{tt}$ and $\Phi_{ttt}$.

\begin{prop}
Under the assumptions of Theorem \ref{main-th-v-2},
\begin{equation}
\|\Phi_{tt}\|_{L^p L^q}+\|\Phi_{tt}\|_{L^\infty \dot{H}^{1}}+\|\Phi_{ttt}\|_{L^\infty L^2}\lesssim 1
\label{Phitt-h1}
\end{equation}
holds true for all pairs $(p,q)$ satisfying \eqref{pq-h2}.
\label{prop-Phitt-h1}
\end{prop}
\begin{proof}
The argument is a carbon copy of the one for Proposition \ref{prop-Phit-h1}, in the sense that we start by writing down the Strichartz estimates \eqref{Str-gen} for the equation \eqref{Box-Phitt}, i.e., 
\begin{equation*}
\aligned
\|\Phi_{tt}\|&_{L^pL^q(I\times \R^5)}+\|\Phi_{tt}\|_{L^\infty\dot{H}^1(I\times \R^5)}+\|\Phi_{ttt}\|_{L^\infty L^2(I\times \R^5)}\\
&\qquad\lesssim \|\Phi_{tt}(a)\|_{\dot{H}^1( \R^5)}+\|\Phi_{ttt}(a)\|_{L^2(\R^5)}+\|\Box\Phi_{tt}\|_{L^1L^2(I\times \R^5)},
\endaligned
\end{equation*}
which are valid in the same context as the one for which \eqref{Phit-h1-str} holds true. Following this, we apply \eqref{bd-box-phitt}, the Sobolev embeddings \eqref{Sob-gen}, \eqref{Phit-h1}, and \eqref{ta-103} to deduce
\begin{equation*}
\aligned
\|\Box\Phi_{tt}\|_{L^1L^2(I\times \R^5)}
\lesssim\, &\|\Phi_t\|_{L^2L^{10/3}(I\times \R^5)}\|\Phi_t\|_{L^2L^5(I\times \R^5)}\\&+ \|\tA(r,v)-1\|_{L^2L^{10/3}(I\times \R^5)}\|\Phi_{tt}\|_{L^2L^5(I\times \R^5)}\\
\lesssim\, &|I|^{1/2}\left(\|\Phi_t\|_{L^\infty H^{1}(I\times \R^5)}+ \|\Phi_{tt}\|_{L^2L^5(I\times \R^5)}\right)\\
\lesssim\, &1+ |I|^{1/2}\|\Phi_{tt}\|_{L^2L^5(I\times \R^5)}.
\endaligned
\end{equation*}
This leads to 
\[
\aligned
\|\Phi_{tt}\|_{L^2L^{5}(I\times \R^5)}+\|\Phi_{tt}\|_{L^\infty\dot{H}^{1}(I\times \R^5)}&+\|\Phi_{ttt}\|_{L^\infty L^2(I \times \R^5)}\\
&\lesssim 1+ \|\Phi_{tt}(a)\|_{\dot{H}^{1}( \R^5)}+\|\Phi_{ttt}(a)\|_{L^2(\R^5)},
\endaligned
\]
where $|I|\sim 1$, yet small enough. Subsequently, by also invoking the Appendix, we derive 
\[
\|\Phi_{tt}\|_{L^2L^{5}}+\|\Phi_{tt}\|_{L^\infty\dot{H}^1}+\|\Phi_{ttt}\|_{L^\infty L^2}\lesssim 1,
\]
which suffices to claim that
\[
\|\Phi_{tt}\|_{L^pL^q}\lesssim 1
\]
holds true for pairs $(p,q)\neq (2,5)$ satisfying \eqref{pq-h2}.
\end{proof}

\subsection{$\dot{H}^3$ and $\dot{H}^2$ regularities for  $\Phi$ and $\Phi_{t}$ and further improvement of the decay information}

As an immediate consequence of the previous proposition, we obtain the corresponding Sobolev regularity for  $\Phi_{t}$.

\begin{prop}
Under the assumptions of Theorem \ref{main-th-v-2}, it is true that
\begin{equation}
\|\Phi_t\|_{L^\infty\dot{H}^2}\lesssim 1. 
\label{Phit-lih2}
\end{equation}
\end{prop}
\begin{proof}
We infer directly on the basis of \eqref{Box-Phit}, \eqref{ta-2}, \eqref{Phitt-h1}, \eqref{decay-A-2}, and \eqref{Phit-h1} that
\[
\aligned
\|\Phi_t\|_{L^\infty\dot{H}^2}\sim \|\Delta\Phi_t\|_{L^\infty L^2}&\lesssim \|\Phi_{ttt}\|_{L^\infty L^2}+\|\Box\Phi_t\|_{L^\infty L^2}\\
&\lesssim 1+\|\tA(r,v)-1\|_{L^\infty L^{5}}\|\Phi_t\|_{L^\infty L^{10/3}}\\
&\lesssim 1.
\endaligned
\] 
\end{proof}

In a similar, but more involved fashion, we derive the $\dot{H}^3$ regularity for  $\Phi$.

\begin{prop}
Under the assumptions of Theorem \ref{main-th-v-2}, it is true that
\begin{equation}
\|\Phi\|_{L^\infty\dot{H}^3}\lesssim 1. 
\label{Phi-lih3}
\end{equation}
\end{prop}
\begin{proof}
We start by arguing that, due to \eqref{Phitt-h1},
\begin{equation}
\aligned
\|\Phi\|_{L^\infty\dot{H}^3}\sim \|D\Delta\Phi\|_{L^\infty L^2}&\lesssim \|D\Phi_{tt}\|_{L^\infty L^2}+\|D\Box\Phi\|_{L^\infty L^2}\\
&\sim \|\Phi_{tt}\|_{L^\infty \dot{H}^1}+\|\partial_r\Box\Phi\|_{L^\infty L^2}\\
&\lesssim 1+\|\partial_r\Box\Phi\|_{L^\infty L^2}.
\endaligned
\label{dr-box-phi-l2}
\end{equation}
A direct computation using \eqref{Box-Phi} yields
\begin{equation*}
\aligned
\partial_r\Box\Phi=\,&\frac{1}{4}\int_0^v \left\{\left(9\tA^{1/2}-3\tA^{-1/2}-\tA^{-3/2}+3\tA^{-5/2}\right)\tA_r\right\}dy\\&+\left(\tA(r,v)-1\right)\left(\frac{3}{2}+\frac{\tA^{-2}(r,v)}{2}\right)\tA^{1/2}(r,v)v_r+\frac{\varphi_{\geq 1/2}}{r^3}
\endaligned
\end{equation*}
and, taking into account the formula \eqref{ta-formula}, we deduce that
\begin{equation}
\left|\partial_r\Box\Phi\right|\lesssim \int_0^{|v|} \left\{\tA^{1/2}|\tA_r|\right\}dy+|\tA(r,v)-1|\tA^{1/2}(r,v)|v_r|+\frac{|\varphi_{\geq 1/2}|}{r^3}.
\label{dr-box-phi}
\end{equation}
It is easy to check that
\begin{equation}
\left\|\frac{\varphi_{\geq 1/2}}{r^3}\right\|_{L^\infty L^2}\lesssim 1
\label{phi-g12}
\end{equation}
and, as a result, we turn our attention to the other two terms on the right-hand side of the previous bound.

For the integral term, we rely on \eqref{ta-formula}, \eqref{ta-r}-\eqref{ta-r-l1}, and \eqref{decay-v-2} to infer that
\begin{equation}
\int_0^{|v|}\left\{\tA^{1/2}|\tA_r|\right\}dy\,\lesssim \,\frac{|v|+v^2}{r^2}\lesssim \frac{1}{r^4}
\label{int-tar-g1}
\end{equation}
holds true when $r\geq 1$, and
\begin{equation}
\int_0^{|v|}\left\{\tA^{1/2}|\tA_r|\right\}dy\,\lesssim \,r(|v|^5+v^6)\lesssim \frac{1}{r^{1/2}}
\label{int-tar-l1}
\end{equation}
is valid when $r<1$. Hence, it is immediate to claim that
\begin{equation}
\left\|\int_0^{|v|}\left\{\tA^{1/2}|\tA_r|\right\}dy\right\|_{L^\infty L^2}\lesssim 1.
\label{int-tar}
\end{equation}

Finally, for the term in \eqref{dr-box-phi} involving $v_r$, we obtain from \eqref{Phi-final} that
\begin{equation}
\Phi_r=\tA^{1/2}(r,v)v_r+\frac{1}{2}\int_0^v\left\{\tA^{-1/2}\tA_r\right\}dy+\frac{\varphi_{\geq 1/2}}{r^3}.
\label{phir-vr}
\end{equation}
It is straightforward to argue that
\[
\left\|\frac{\varphi_{\geq 1/2}}{r^3}\right\|_{L^\infty L^{10/3}}\lesssim 1
\]
and, using \eqref{int-tar-g1} and \eqref{int-tar-l1}, we also have that
\[
\left\|\int_0^{v}\left\{\tA^{-1/2}\tA_r\right\}dy\right\|_{L^\infty L^{10/3}}\lesssim 1.
\]
Furthermore, due to the Sobolev embeddings \eqref{Sob-gen} and \eqref{Phi-lih2}, we derive
\[
\|\Phi_r\|_{L^\infty L^{10/3}}\lesssim \|\Phi_r\|_{L^\infty H^{1}}\lesssim 1. 
\]
Therefore, as the combined result of the last four mathematical statements, we infer that 
\[
\|\tA^{1/2}(r,v)v_r\|_{L^\infty L^{10/3}}\lesssim 1.
\]
If we couple this estimate with \eqref{decay-A-2}, we conclude that
\[
\|(\tA(r,v)-1)\tA^{1/2}(r,v)v_r\|_{L^\infty L^2}\lesssim \|\tA(r,v)-1\|_{L^\infty L^5}\|\tA^{1/2}(r,v)v_r\|_{L^\infty L^{10/3}}\lesssim 1.
\]
Together with \eqref{dr-box-phi-l2}, \eqref{dr-box-phi}, \eqref{phi-g12}, and \eqref{int-tar}, this bounds yields \eqref{Phi-lih3}, thus finishing the proof.
\end{proof}

As a corollary of this result, we can argue as in the section devoted to the $H^2$-type analysis and further improve the decay estimates for $\Phi$, $v$, and $\tA(r,v)-1$.

\begin{prop}
Under the assumptions of Theorem \ref{main-th-v-2}, we have
\begin{align}
|\Phi(t,r)|&\lesssim \frac{1}{1+r^2},\label{decay-Phi-3}\\
|v(t,r)|&\lesssim \frac{1}{1+r^2},\label{decay-v-3}\\
|\tA(r,v)-1|&\lesssim \frac{1}{1+r^4}.\label{decay-A-3}
\end{align}
\end{prop}

\begin{remark}
We notice that the decay bound \eqref{decay-v-3} easily implies the portion of the main estimate to be proved about $v$ (i.e., \eqref{livr}), which doesn't involve its gradient:
\begin{equation}
\|(1+r)|v|\|_{L^\infty_{t,x}}\leq  \|(1+r^2)|v|\|_{L^\infty_{t,x}}\lesssim 1.
\label{li-v}
\end{equation}
\end{remark}

\begin{remark}
In what concerns $\nabla v$, we can show with the facts obtained so far that
\begin{equation}
\|(1+r)|\nabla v|\|_{L^\infty_{t,x}([0,T)\times \{r\geq 1\})}\lesssim 1
\label{li-nabla-v-1}
\end{equation}
holds true. First, we can rewrite \eqref{phir-vr} in the form
\begin{equation}
\aligned
v_r\,=\,\tA^{-1/2}(r,v)\left(\Phi_r-\frac{1}{2}\int_{0}^{v}\left\{\tA^{-1/2}\tA_r\right\}dy-\frac{\varphi_{\geq 1/2}}{r^3}\right),
\endaligned
\label{vr-phir}
\end{equation}
and we can infer from \eqref{Phit-final} that
\begin{equation}
v_t\,=\,\tA^{-1/2}(r,v)\Phi_t.
\label{vt-phit}
\end{equation}
Thus, in the regime when $r\geq 1$, we deduce based on \eqref{ta-r}, \eqref{ta-r-g1}, \eqref{decay-A}, and \eqref{decay-v} that
\[
\aligned
r(|v_r|+|v_t|)&\lesssim r\left(|\Phi_r|+|\Phi_t|+\int_0^{|v|}\left\{\frac{1+|y|}{r^2}\right\}\,dy+\frac{\left|\varphi_{\geq 1/2}\right|}{r^3}\right)\\
&\lesssim r\left(|\Phi_r|+|\Phi_t|+\frac{|v|+v^2}{r^2}+\frac{1}{r^3}\right)\\
&\lesssim r\left(|\Phi_r|+|\Phi_t|+\frac{1}{r^3}\right).
\endaligned
\]
Finally, with the help of \eqref{rad-Sob-1} and \eqref{H3-norm}, we derive
\[
r(|v_r|+|v_t|)\lesssim \|\Phi_r\|_{L^\infty\dot{H}^{3/2}} +  \|\Phi_t\|_{L^\infty\dot{H}^{3/2}} +1\lesssim 1,
\]
which yields \eqref{li-nabla-v-1}.
\end{remark}


\section{Final estimates and conclusion of the argument}

On the basis of the previous two remarks, specifically the estimates \eqref{li-v} and \eqref{li-nabla-v-1}, in order to conclude the argument for \eqref{livr} (and thus finish the proof of Theorem \ref{main-th-v-2}), we are left to show that   
\begin{equation}
\|\nabla v\|_{L^\infty_{t,x}([0,T)\times \{r< 1\})}\lesssim 1
\label{li-nabla-v-2}
\end{equation}
is valid. As in the derivation of \eqref{li-nabla-v-1}, we start by relying on \eqref{vr-phir}, \eqref{vt-phit}, \eqref{ta-r}, \eqref{ta-r-l1}, \eqref{decay-A}, and \eqref{decay-v-3} to infer that
\[
\aligned
|v_r|+|v_t|&\lesssim |\Phi_r|+|\Phi_t|+\int_0^{|v|}\left\{ry^4\right\}\,dy+\frac{\left|\varphi_{\geq 1/2}\right|}{r^3}\\
&\lesssim |\Phi_r|+|\Phi_t|+r|v|^5+1\\
&\lesssim |\Phi_r|+|\Phi_t|+1
\endaligned
\]
holds true when $r<1$. This implies
\begin{equation*}
\|\nabla v\|_{L^\infty_{t,x}([0,T)\times \{r< 1\})}\lesssim \|\Phi_r\|_{L^\infty_{t,x}} +  \|\Phi_t\|_{L^\infty_{t,x}} +1
\end{equation*}
and, with the help of \eqref{Phi-h1} and the classical Sobolev embedding \eqref{Sob-classic}, we obtain the desired bound if we show that
\begin{equation}
\|\Phi\|_{L^\infty \dot{H}^{s}}+\|\Phi_t\|_{L^\infty \dot{H}^{s-1}}\lesssim 1
\label{phi-phit-hs}
\end{equation}
is valid.

The strategy is to prove first the finiteness of the norm involving $\Phi_t$ by using energy estimates applied to equation \eqref{Box-Phit}. As a byproduct of this argument, we also get to control $\|\Phi_{tt}\|_{L^\infty \dot{H}^{s-2}}$, which, coupled to equation \eqref{Box-Phi} satisfied by $\Phi$, allows us to deduce the finiteness of $\|\Phi\|_{L^\infty \dot{H}^{s}}$ and thus finish the proof of \eqref{phi-phit-hs}. 

\vspace{.1in}
\hrule


\subsection{New qualitative bounds for $v$ and $\tA(r,v)$}

As one can imagine from the structure of the equations involved in this step of the main argument (i.e., \eqref{Box-Phi} and \eqref{Box-Phit}), it is important to derive more qualitative information on $v$ and $\tA(r,v)$, in addition to  \eqref{v-h1}, \eqref{decay-v-3}, and \eqref{decay-A-3}. First, we prove the following result.

\begin{prop}
Under the assumptions of Theorem \ref{main-th-v-2}, we have
\begin{align}
\|\nabla v\|_{L^\infty L^4}+&\|\Delta v\|_{L^\infty L^2}\lesssim 1,\label{nabla-delta-v}\\
\|\nabla (\tA(r,v))\|_{L^\infty L^4}+&\|\Delta( \tA(r,v))\|_{L^\infty L^2}\lesssim 1\label{nabla-delta-A}.
\end{align}
\end{prop}

\begin{proof}
We start by arguing for the finiteness of both $L^\infty L^4$ norms. Using \eqref{tA-Phit}, the Sobolev embeddings \eqref{Sob-gen}, \eqref{phitlil2}, and \eqref{Phit-lih2}, we immediately deduce
\begin{equation}
\| \partial_t(\tA(r,v))\|_{L^\infty L^4}\lesssim \|\Phi_t\|_{L^\infty L^4}\lesssim \|\Phi_t\|_{L^\infty H^{5/4}}\lesssim 1,
\label{tA-lil4}
\end{equation}
which, together with \eqref{Phit-final}, implies
\begin{equation}
\| v_t\|_{L^\infty L^4}\lesssim \|\Phi_t\|_{L^\infty L^4}\lesssim 1.
\label{vt-lil4}
\end{equation}

Next, if we use \eqref{vr-phir} jointly with \eqref{ta-r}-\eqref{ta-r-l1}, then we obtain
\begin{equation*}
|v_r|\lesssim |\Phi_r|+\frac{\left|\varphi_{\geq 1/2}\right|}{r^3}+
\begin{cases}
\frac{|v|+v^2}{r^2}, \, &r\geq 1,\\
r|v|^5, \, &r< 1.
\end{cases}
\end{equation*}
Consequently, by also factoring in the Sobolev embeddings \eqref{Sob-gen}, \eqref{decay-v-3}, \eqref{philil2}, and \eqref{Phi-lih3}, it follows that
\begin{equation}
\aligned
\| v_r\|_{L^\infty L^4}\lesssim &\,\|\Phi_r\|_{L^\infty L^4}+\left\|\frac{1}{r^3}\right\|_{L^\infty L^4([0,T)\times \{r\geq 1/2\})}\\&+ \left\|\frac{1}{r^4}\right\|_{L^\infty L^4([0,T)\times \{r\geq 1\})}+ \left\|r\right\|_{L^\infty L^4([0,T)\times \{r< 1\})}\\\lesssim &\, \|\Phi\|_{L^\infty H^{9/4}}+1\\ \lesssim &\, 1.
\endaligned
\label{vr-lil4}
\end{equation}
To conclude this part of the argument, if we rely on the formula \eqref{ta-formula}, \eqref{ta-r}-\eqref{ta-r-l1}, \eqref{decay-v-3}, and the previous estimate, then we can infer that
\begin{equation*}
\aligned
\| \partial_r&(\tA(r,v))\|_{L^\infty L^4}\\&\lesssim \left\|\frac{1+|v|}{r^2}+\frac{|v_r|}{r}\right\|_{L^\infty L^4([0,T)\times \{r\geq 1\})}+ \left\|rv^4+|v||v_r|\right\|_{L^\infty L^4([0,T)\times \{r< 1\})}\\&\lesssim \left\|\frac{1}{r^2}\right\|_{L^\infty L^4([0,T)\times \{r\geq 1\})}+ \left\|r\right\|_{L^\infty L^4([0,T)\times \{r< 1\})}+\left\|v_r\right\|_{L^\infty L^4}\\ &\lesssim 1.
\endaligned
\end{equation*}

Next, we prove the finiteness of the second norm in \eqref{nabla-delta-v} by showing that
\begin{equation}
\|v_{tt}\|_{L^\infty L^2}+\|\Box v\|_{L^\infty L^2}\lesssim 1.
\label{vtt-dv}
\end{equation}
We deduce directly from \eqref{Phitt-final} that
\[
|v_{tt}|\lesssim |\Phi_{tt}|+v_t^2
\]
and, subsequently, on the basis of \eqref{Phit-h1} and \eqref{vt-lil4}, we infer that
\[
\|v_{tt}\|_{L^\infty L^2}\lesssim \|\Phi_{tt}\|_{L^\infty L^2}+\|v_t\|_{L^\infty L^4}^2\lesssim 1.
\]

For estimating the $L^\infty L^2$ norm of $\Box v$, we rely on \eqref{main-v} and analyze separately each term on the right-hand side of that equation. First, we can easily argue by using the definitions of $\varphi$ and $\varphi_{>1}$ and \eqref{v-l2} that
\begin{equation*}
\left\| \frac{1}{r}\Delta_3 \varphi\right\|_{L^\infty L^2}+\left\|\frac{2}{r^2}\varphi_{>1} v\right\|_{L^\infty L^2} \lesssim \left\|\frac{1}{r}\right\|_{L^\infty L^2([0,T)\times \{1\leq r\leq 2\})}+\|v\|_{L^\infty L^2}\lesssim 1.
\end{equation*}
Next, due to \eqref{main-2}, \eqref{utov}, and \eqref{decay-v-3}, we derive
\begin{equation*}
\aligned
\frac{1}{r}\left|\varphi_{>1}N(r, rv+\varphi, \nabla(rv+\varphi))\right|&\lesssim \frac{1}{r}|\varphi_{>1}| \left(\frac{1+|\nabla(rv+\varphi)|^2}{r^2}+\frac{|v+rv_r+\varphi_r|}{r^3}\right)\\
&\lesssim |\varphi_{>1}| \left(\frac{1}{r^3}+\frac{|\nabla v|^2}{r}+\frac{|v|}{r^4}+\frac{|v_r|}{r^3}\right).
\endaligned
\end{equation*}
Hence, by applying \eqref{v-h1}, \eqref{vt-lil4}, and \eqref{vr-lil4}, we obtain that
\begin{equation*}
\aligned
&\left\|\frac{1}{r}\varphi_{>1}N(r, rv+\varphi, \nabla(rv+\varphi))\right\|_{L^\infty L^2}\\
&\qquad\qquad\lesssim \left\|\frac{1}{r^3}\right\|_{L^\infty L^2([0,T)\times \{ r\geq 1/2\})}+\|\nabla v\|^2_{L^\infty L^4}+\|v\|_{L^\infty L^2}+\|v_r\|_{L^\infty L^2}\\&\qquad\qquad\lesssim 1.
\endaligned
\end{equation*}
Finally, on the basis of \eqref{Nv-l1}, we deduce
\begin{equation*}
\left|\varphi_{<1}\left(\frac{1}{r}N(r, rv, \nabla(rv))+\frac{2}{r^2}v\right)\right|\lesssim |\varphi_{<1}|\left(|v|^3+|v|^5+|v||\nabla v|^2+rv^4|v_r|\right)
\end{equation*}
and, consequently, 
\begin{equation*}
\aligned
&\left\|\varphi_{<1}\left(\frac{1}{r}N(r, rv, \nabla(rv))+\frac{2}{r^2}v\right)\right\|_{L^\infty L^2}\\
&\qquad \qquad\qquad\lesssim \left(\|v\|_{L^\infty_{t,x}}^2+\|v\|_{L^\infty_{t,x}}^4 \right)\|v\|_{L^\infty L^2}+\|v\|_{L^\infty_{t,x}}\|\nabla v\|^2_{L^\infty L^4}\\&\qquad \qquad\qquad\quad+\|v\|_{L^\infty_{t,x}}^{7/2}\|v\|^{1/2}_{L^\infty L^2}\|\nabla v\|_{L^\infty L^4}.
\endaligned
\end{equation*}
The desired estimate follows as a result of \eqref{decay-v-3}, \eqref{v-l2}, \eqref{vt-lil4}, and \eqref{vr-lil4}, and this concludes the argument for the finiteness of $\|\Box v\|_{L^\infty L^2}$. Thus, we have finished the proof of \eqref{nabla-delta-v}.

A direct computation based on \eqref{ta-formula} reveals that
\begin{equation}
\aligned
\Delta (\tA(r,v))=\ & 4\,\frac{\cos(2u)}{r^2}\left(v+rv_r+\varphi_r\right)^2+2\,\frac{\sin(2u)}{r^2}\left(r\Delta v-2v_r+\varphi_{rr}\right)\\&-4\,\frac{\sin^2u}{r^4}.
\endaligned
\label{delta-ta}
\end{equation}
If we invoke \eqref{Hardy}, the definition of $\varphi$, and \eqref{nabla-delta-v}, then we infer that
\begin{equation}
\aligned
\left\|\frac{\cos(2u)}{r^2}\left(v+rv_r+\varphi_r\right)^2\right\|_{L^\infty L^2}&\lesssim \left\|\frac{v}{r}\right\|^2_{L^\infty L^4}+\|v_r\|^2_{L^\infty L^4}+\left\|\frac{\varphi_r}{r}\right\|^2_{L^\infty L^4}\\&\lesssim \|\nabla_x v\|^2_{L^\infty L^4}+\left\|\frac{1}{r}\right\|^2_{L^\infty L^4([0,T)\times \{1\leq r\leq 2\})}\\&\lesssim 1.
\endaligned
\label{delta-ta-1}
\end{equation}
Next, using the elementary bound
\[
\frac{|\sin(2u)|}{r}\lesssim  \tA^{1/2}(r,v),
\] 
we derive 
\[
\left|\frac{\sin(2u)}{r^2}\left(r\Delta v-2v_r+\varphi_{rr}\right)\right|\lesssim  \tA^{1/2}(r,v)\left(|\Delta v|+\frac{|\nabla_xv|+|\varphi_{rr}|}{r}\right).
\]
Hence, due to \eqref{Hardy}, the definition of $\varphi$, \eqref{decay-A-3}, and \eqref{nabla-delta-v}, we obtain
\begin{equation}
\aligned
&\left\|\frac{\sin(2u)}{r^2}\left(r\Delta v-2v_r+\varphi_{rr}\right)\right\|_{L^\infty L^2}\\
&\qquad\quad\lesssim \left\| \tA^{1/2}(r,v)\right\|_{L^\infty_{t,x}}\left(\left\|\Delta v\right\|_{L^\infty L^2}+\left\|\frac{\nabla_xv}{r}\right\|_{L^\infty L^2}+\left\|\frac{\varphi_{rr}}{r}\right\|_{L^\infty L^2}\right)\\&\qquad\quad\lesssim  \left\| \tA^{1/2}(r,v)\right\|_{L^\infty_{t,x}}\left(\left\|\Delta v\right\|_{L^\infty L^2}+\left\|\frac{1}{r}\right\|_{L^\infty L^2([0,T)\times \{1\leq r\leq 2\})}\right)\\&\qquad\quad\lesssim 1.
\endaligned
\label{delta-ta-2}
\end{equation}
Following this, one more application of \eqref{decay-A-3} yields
\begin{equation*}
\left\|\frac{\sin^2u}{r^4}\right\|_{L^\infty L^2}\sim\left\|\frac{\tA(r,v)-1}{r^2}\right\|_{L^\infty L^2} \lesssim \left\|\frac{1}{(1+r^4)r^2}\right\|_{L^\infty L^2}\lesssim 1.
\end{equation*}
Jointly with \eqref{delta-ta}-\eqref{delta-ta-2}, this estimate implies 
\begin{equation*}
\|\Delta (\tA(r,v))\|_{L^\infty L^2}\lesssim 1,
\end{equation*}
which concludes the argument for \eqref{nabla-delta-A} and the whole proof of this proposition.
\end{proof}

Next, due to the presence of $\tA^{-2}$ on the right-hand side of \eqref{Box-Phit}, we also need estimates for derivatives of $\tA^{-1}$.

\begin{prop}
Under the assumptions of Theorem \ref{main-th-v-2}, the fixed-time bound
\begin{equation}
\|D^\sigma(\tA^{-1}(r,v))\|_{L^p(\R^5)}\lesssim \|D^\sigma(\tA(r,v))\|_{L^p(\R^5)}
\label{ta-inv-ta}
\end{equation}
holds true uniformly on $[0,T)$ for all $1<\sigma<2$ and $1<p<\infty$. 
\end{prop}
\begin{proof}
We start by applying the Kato-Ponce type inequalities \eqref{Lbnz-1} and \eqref{Lbnz-2} to deduce
\begin{equation*}
\aligned
\|D^\sigma(\tA\,\tA^{-1}(r,v))&-D^\sigma(\tA(r,v))\,\tA^{-1}(r,v)- D^\sigma(\tA^{-1}(r,v))\,\tA(r,v) \|_{L^p(\R^5)}\\
&\lesssim \|D^{\sigma/2}(\tA(r,v))\|_{L^{2p}(\R^5)}\|D^{\sigma/2}(\tA^{-1}(r,v))\|_{L^{2p}(\R^5)}
\endaligned
\end{equation*}
and 
\begin{equation*}
\aligned
\|D^{\sigma/2}(\tA\,\tA^{-1}(r,v))&- D^{\sigma/2}(\tA^{-1}(r,v))\,\tA(r,v)\|_{L^{2p}(\R^5)}\\
&\lesssim \|D^{\sigma/2}(\tA(r,v))\|_{L^{2p}(\R^5)}\|\tA^{-1}(r,v)\|_{L^{\infty}(\R^5)}.
\endaligned
\end{equation*}
Following this, due to $\tA\,\tA^{-1}(r,v)\equiv 1$ and $\sigma>0$, we notice that 
\begin{equation*}
D^\sigma(\tA\,\tA^{-1}(r,v))=D^{\sigma/2}(\tA\,\tA^{-1}(r,v))\equiv 1.
\end{equation*}
Moreover, it is easily seen that  \eqref{ta-formula} implies
\begin{equation}
\|\tA^{-1}(r,v)\|_{L^{\infty}(\R^5)}\lesssim 1.
\label{ta-inv-li}
\end{equation}
Using these two observations jointly with the previous two bounds and \eqref{decay-A-3}, we derive
\begin{equation*}
\aligned
\|D^\sigma(\tA^{-1}(r,v))\|_{L^p(\R^5)}\lesssim\, &\|D^\sigma(\tA(r,v))\|_{L^p(\R^5)}\\&+ \|D^{\sigma/2}(\tA(r,v))\|_{L^{2p}(\R^5)}\|D^{\sigma/2}(\tA^{-1}(r,v))\|_{L^{2p}(\R^5)}
\endaligned
\end{equation*}
and
\begin{equation*}
\aligned
\| D^{\sigma/2}(\tA^{-1}(r,v))\|_{L^{2p}(\R^5)}
\lesssim \|D^{\sigma/2}(\tA(r,v))\|_{L^{2p}(\R^5)}.
\endaligned
\end{equation*}
Consequently, we infer that
\begin{equation*}
\aligned
\|D^\sigma(\tA^{-1}(r,v))\|_{L^p(\R^5)}\lesssim\,\|D^\sigma(\tA(r,v))\|_{L^p(\R^5)}+ \|D^{\sigma/2}(\tA(r,v))\|^2_{L^{2p}(\R^5)}.
\endaligned
\end{equation*}
Finally, combining this estimate with an application of the interpolation inequality \eqref{interpol-bd} that yields
\begin{equation*}
\aligned
\|D^{\sigma/2}(\tA(r,v))\|^2_{L^{2p}(\R^5)}&\lesssim \|D^\sigma(\tA(r,v))\|_{L^p(\R^5)}\|\tA(r,v)\|_{L^{\infty}(\R^5)}\\
&\lesssim \|D^\sigma(\tA(r,v))\|_{L^p(\R^5)},
\endaligned
\end{equation*}
we reach the desired conclusion.
\end{proof}


\subsection{Improved Sobolev regularities for $\Phi_t$ and $\Phi$}
Now, we have all the prerequisites needed to upgrade the $H^2$ and $H^3$ regularities for  $\Phi_t$ and $\Phi$, respectively, to the level of the ones featured in \eqref{phi-phit-hs}. As outlined at the start of this section, we first focus our analysis on $\Phi_t$.

\begin{prop}
Under the assumptions of Theorem \ref{main-th-v-2}, with $s>7/2$ replaced by $7/2<s<18/5$,
\begin{equation}
\|\Phi_t\|_{L^\infty \dot{H}^{s-1}}+\|\Phi_{tt}\|_{L^\infty \dot{H}^{s-2}}\lesssim 1
\label{phit-phitt-hs}
\end{equation}
is valid.
\label{prop-phit}
\end{prop}

\begin{proof}
We commence by relying on the energy-type estimate \eqref{en-hs} applied to \eqref{Box-Phit} to argue that
\begin{equation}
\aligned
\|\Phi_t\|_{L^\infty \dot{H}^{s-1}}+\|\Phi_{tt}\|_{L^\infty \dot{H}^{s-2}}\lesssim\, &\|\Phi_t(0)\|_{\dot{H}^{s-1}( \R^5)}+\|\Phi_{tt}(0)\|_{\dot{H}^{s-2}(\R^5)}\\&+\|\Box\Phi_t\|_{L^1\dot{H}^{s-2}}.
\endaligned
\end{equation}
The Appendix ensures that
\begin{equation*}
\|\Phi_t(0)\|_{\dot{H}^{s-1}( \R^5)}+\|\Phi_{tt}(0)\|_{\dot{H}^{s-2}(\R^5)}\lesssim 1
\end{equation*}
and, hence, in order to deduce \eqref{phit-phitt-hs}, it is enough to show that
\begin{equation}
\|\Box\Phi_t\|_{L^1\dot{H}^{s-2}}\lesssim 1.
\label{box-phit-hs-1}
\end{equation}
Using \eqref{Box-Phit} and the fractional Leibniz bound \eqref{Lbnz-0}, we derive that
\begin{equation}
\aligned
\|\Box\Phi_t\|&_{L^1\dot{H}^{s-2}}\\
\lesssim\,&\|D^{s-2}(\tA(r,v))\|_{L^\infty L^{4/(s-2)}}\left(1+\|\tA^{-1}(r,v)\|^2_{L^\infty_{t,x}}\right)\|\Phi_t\|_{L^\infty L^{4/(4-s)}}\\
&+\|\tA(r,v)-1\|_{L^\infty_{t,x}}\|D^{s-2}(\tA^{-2}(r,v))\|_{L^\infty L^{4/(s-2)}}\|\Phi_t\|_{L^\infty L^{4/(4-s)}}\\
&+\|\tA(r,v)-1\|_{L^\infty_{t,x}}\left(1+\|\tA^{-1}(r,v)\|^2_{L^\infty_{t,x}}\right)\|D^{s-2}\Phi_t\|_{L^\infty L^2}.
\endaligned\label{box-phit-hs-2}
\end{equation}

First, we deal with the norms involving $\Phi_t$, for which an application of the Sobolev embeddings \eqref{Sob-gen} produces
\begin{equation*}
\|\Phi_t\|_{L^\infty L^{4/(4-s)}}\lesssim \|\Phi_t\|_{L^\infty H^{5(s-2)/4}}\lesssim \|\Phi_t\|_{L^\infty H^{2}},
\end{equation*}
since $s<18/5$. This also implies
\begin{equation*}
\|D^{s-2}\Phi_t\|_{L^\infty L^2}\sim \|\Phi_t\|_{L^\infty \dot{H}^{s-2}}\lesssim \|\Phi_t\|_{L^\infty H^{8/5}}.
\end{equation*} 
Therefore, due to \eqref{phitlil2} and \eqref{Phit-lih2}, we obtain
\begin{equation*}
\|\Phi_t\|_{L^\infty L^{4/(4-s)}}+\|D^{s-2}\Phi_t\|_{L^\infty L^2}\lesssim 1.
\end{equation*}

Next, we work on the norms depending on $\tA$ and $\tA^{-1}$. On account of \eqref{Lbnz-0}, \eqref{ta-inv-ta}, and \eqref{ta-inv-li}, we infer that
\begin{equation*}
\aligned
\|D^{s-2}(\tA^{-2}(r,v))\|_{L^\infty L^{4/(s-2)}}&\lesssim \|D^{s-2}(\tA^{-1}(r,v))\|_{L^\infty L^{4/(s-2)}}\|\tA^{-1}(r,v)\|_{L^\infty_{t,x}}\\&\lesssim \|D^{s-2}(\tA(r,v))\|_{L^\infty L^{4/(s-2)}}.
\endaligned
\end{equation*}
Using the interpolation inequality \eqref{interpol-bd} and \eqref{nabla-delta-A}, we deduce
\begin{equation*}
\|D^{s-2}(\tA(r,v))\|_{L^\infty L^{4/(s-2)}}\lesssim \|\nabla_x(\tA(r,v))\|^{4-s}_{L^\infty L^4}\|\Delta(\tA(r,v))\|^{s-3}_{L^\infty L^2}\lesssim 1.
\end{equation*}
Finally, by invoking \eqref{decay-A-3} and \eqref{ta-inv-li}, we also control the $L^\infty_{t,x}$ norms in \eqref{box-phit-hs-2} and, thus, the argument for \eqref{box-phit-hs-1} is concluded.
\end{proof}

Following this, we can finish the proof of \eqref{phi-phit-hs} and, consequently, the proof of our main result by deriving the expected Sobolev regularity for $\Phi$. 

\begin{prop}
Under the assumptions of Proposition \ref{prop-phit}, we have
\begin{equation}
\|\Phi\|_{L^\infty \dot{H}^{s}}\lesssim 1.
\label{phi-hs+1}
\end{equation}
\end{prop}
\begin{proof}
We start the argument by using \eqref{phit-phitt-hs} and $7/2<s<18/5$ to argue that
\begin{equation}
\aligned
\|\Phi\|_{L^\infty \dot{H}^{s}}\sim \|\Delta\Phi\|_{L^\infty \dot{H}^{s-2}}&\leq \|\Phi_{tt}\|_{L^\infty \dot{H}^{s-2}}+\|\Box\Phi\|_{L^\infty \dot{H}^{s-2}}\\
&\lesssim 1+\|\Box\Phi\|_{L^\infty H^{2}}.
\endaligned
\label{phi-hs+1-v2}
\end{equation}
On one hand, a joint application of \eqref{rll1}, \eqref{rgrt1}, and \eqref{decay-v-3} yields
\begin{equation}
\aligned
\|\Box&\Phi\|_{L^\infty L^{2}}\\&\lesssim \left\||v|^3+v^4\right\|_{L^\infty L^2([0,T)\times \{r<1\})}+\left\|\frac{|\varphi_{\geq 1/2}|}{r^3}+\frac{|v|}{r^2}\right\|_{L^\infty L^2([0,T)\times \{r\geq1\})}\\
&\lesssim \left\|1\right\|_{L^\infty L^2([0,T)\times \{r<1\})}+\left\|\frac{1}{r^3}\right\|_{L^\infty L^2([0,T)\times \{r\geq1\})}\\
&\lesssim 1.
\endaligned
\label{box-phi-l2}
\end{equation}

On the other hand, we observe that
\begin{equation}
\|\Box\Phi\|_{L^\infty \dot{H}^{2}}\sim\|\Delta\Box\Phi\|_{L^\infty L^{2}}
\label{box-phi-h2}
\end{equation}
and, subsequently, a tedious but direct computation based on \eqref{Box-Phi} leads to
\begin{equation*}
\aligned
\Delta\Box\Phi=\,&\frac{1}{4}\int_{0}^{v}\left\{9\tA^{1/2}-3\tA^{-1/2}-\tA^{-3/2}+3\tA^{-5/2}\right\}\Delta \tA\ dy\\
&+\frac{1}{8}\int_{0}^{v}\left\{9\tA^{-1/2}+3\tA^{-3/2}+3\tA^{-5/2}-15\tA^{-7/2}\right\}\tA_r^2\ dy\\
&+\frac{1}{2}\left\{9\tA^{1/2}(r,v)-3\tA^{-1/2}(r,v)-\tA^{-3/2}(r,v)+3\tA^{-5/2}(r,v)\right\}\tA_r(r,v)v_r\\
&+\frac{1}{4}\left\{9\tA^{1/2}(r,v)-3\tA^{-1/2}(r,v)-\tA^{-3/2}(r,v)+3\tA^{-5/2}(r,v)\right\}\tA_y(r,v)v^2_r\\
&+\tA^{1/2}(r,v)(\tA(r,v)-1)\left(\frac{3}{2}+\frac{\tA^{-2}(r,v)}{2}\right)\Delta v\\
&+\frac{\varphi_{\geq 1/2}}{r^3}.
\endaligned
\end{equation*}
From previous calculations, we already know that the last term on the right-hand side has a finite $L^\infty L^2$ norm, while \eqref{decay-A-3}, \eqref{nabla-delta-v}, \eqref{decay-v-3}, \eqref{nabla-delta-A}, and
\begin{equation*}
\left|\tA_y(r,v)\right|=2\frac{|\sin (2u)|}{r}\lesssim \tA^{1/2}(r,v)
\end{equation*}
together imply
\begin{equation*}
\aligned
&\left\|\frac{1}{2}\left\{9\tA^{1/2}(r,v)-3\tA^{-1/2}(r,v)-\tA^{-3/2}(r,v)+3\tA^{-5/2}(r,v)\right\}\tA_r(r,v)v_r\right\|_{L^\infty L^2}\\
&\qquad\qquad\qquad\lesssim \|\tA_r(r,v)\|_{L^\infty L^4}\|v_r\|_{L^\infty L^4}\\
&\qquad\qquad\qquad\lesssim \left(\|\partial_r(\tA_r(r,v))\|_{L^\infty L^4}+(1+\|v\|_{L^\infty_{t,x}})\|v_r\|_{L^\infty L^4}\right)\|v_r\|_{L^\infty L^4}\\
&\qquad\qquad\qquad\lesssim 1,
\endaligned
\end{equation*}
\begin{equation*}
\aligned
&\left\|\frac{1}{4}\left\{9\tA^{1/2}(r,v)-3\tA^{-1/2}(r,v)-\tA^{-3/2}(r,v)+3\tA^{-5/2}(r,v)\right\}\tA_y(r,v)v^2_r\right\|_{L^\infty L^2}\\
&\qquad\qquad\qquad\qquad\qquad\lesssim\|v_r\|^2_{L^\infty L^4}\\
&\qquad\qquad\qquad\qquad\qquad\lesssim 1,
\endaligned
\end{equation*}
and
\begin{equation*}
\left\|\tA^{1/2}(r,v)(\tA(r,v)-1)\left(\frac{3}{2}+\frac{\tA^{-2}(r,v)}{2}\right)\Delta v\right\|_{L^\infty L^2}\lesssim \|\Delta v\|_{L^\infty L^2}\lesssim 1.
\end{equation*}

Therefore, we are left to analyze the two integral terms in the expression for $\Delta\Box\Phi$. For the second one, we can use the obvious bound $\tA\geq 1$ and \eqref{ta-r}-\eqref{ta-r-l1} to infer that
\begin{equation*}
\aligned
&\left|\frac{1}{8}\int_{0}^{v}\left\{9\tA^{-1/2}+3\tA^{-3/2}+3\tA^{-5/2}-15\tA^{-7/2}\right\}\tA_r^2\ dy\right|\\
&\qquad\qquad\qquad\qquad\lesssim \int_{0}^{|v|}\tA_r^2\ dy\\
&\qquad\qquad\qquad\qquad\lesssim \begin{cases}
\frac{|v|+|v|^3}{r^4}, \, &r\geq 1,\\
r^2|v|^9, \, &r< 1.
\end{cases}
\endaligned
\end{equation*}
If we now factor in \eqref{decay-v-3}, then we effortlessly obtain
\begin{equation*}
\left\|\frac{1}{8}\int_{0}^{v}\left\{9\tA^{-1/2}+3\tA^{-3/2}+3\tA^{-5/2}-15\tA^{-7/2}\right\}\tA_r^2\ dy\right\|_{L^\infty L^2}\lesssim 1.
\end{equation*}
For the first integral term, a straightforward calculation using \eqref{ta-r} reveals that
\begin{equation*}
\Delta\tA_r=\frac{-4\sin^2(ry+\varphi)}{r^4}+\frac{2\sin2(ry+\varphi)\cdot \varphi_{rr}+ 4\cos2(ry+\varphi)\cdot(y+\varphi_r)^2}{r^2}.
\end{equation*}
An investigation similar to the one producing \eqref{ta-r-g1} and \eqref{ta-r-l1} yields in this case
\begin{equation*}
\left|\Delta\tA_r\right|\lesssim \begin{cases}
\frac{1+y^2}{r^2}, \, &r\geq 1,\\
y^4, \, &r< 1.
\end{cases}
\end{equation*}
Therefore, we deduce
\begin{equation*}
\aligned
&\left|\frac{1}{4}\int_{0}^{v}\left\{9\tA^{1/2}-3\tA^{-1/2}-\tA^{-3/2}+3\tA^{-5/2}\right\}\Delta \tA\ dy\right|\\
&\qquad\qquad\qquad\qquad\lesssim \int_{0}^{|v|}\tA^{1/2}\left|\Delta \tA\right|\ dy\\
&\qquad\qquad\qquad\qquad\lesssim \begin{cases}
\frac{|v|+|v|^3}{r^2}, \, &r\geq 1,\\
|v|^5+v^6, \, &r< 1.
\end{cases}
\endaligned
\end{equation*}
If we rely yet again on \eqref{decay-v-3}, we derive 
\begin{equation*}
\left\|\frac{1}{4}\int_{0}^{v}\left\{9\tA^{1/2}-3\tA^{-1/2}-\tA^{-3/2}+3\tA^{-5/2}\right\}\Delta \tA\ dy\right\|_{L^\infty L^2}\lesssim 1, 
\end{equation*}
which, jointly with the estimates for the other terms in the formula of $\Delta\Box\Phi$, implies
\begin{equation*}
\|\Delta\Box\Phi\|_{L^\infty L^{2}}\lesssim 1.
\end{equation*}
Together with \eqref{phi-hs+1-v2}, \eqref{box-phi-l2}, and \eqref{box-phi-h2}, this bound shows that \eqref{phi-hs+1} is valid and the argument is finished.
\end{proof}


\section*{Appendix}

As mentioned in the introduction, this appendix is devoted to showing that the Sobolev regularity of the initial data appearing as part of the hypothesis in Theorem \ref{main-th-v-2} (equivalent to the one for data in  Theorem \ref{main-th}) implies the claims made throughout the main argument about various norms evaluated at $t=0$.

The first such quantity is the energy \eqref{tote-0}, for which a direct analysis based on Sobolev embeddings, radial Sobolev estimates, and Hardy-type inequalities proves that it is finite if 
\[
(u_0,u_1)\in \left(\dot{H}^{7/4+\epsilon}\cap\dot{H}^1\right)(\R^3)\times L^2(\R^3),
\]
with $\epsilon>0$ being arbitrarily small. Next, in section \ref{sect-en}, a careful reading shows that the arguments there are valid as long as $v(0)\in H^2(\R^5)$. This was used, for example, in \eqref{phir-l2} and it holds true since $v(0)\in H^s(\R^5)$ with $s>7/2$.

Following this, we have three more claims to argue for:
\begin{equation}
\|\Phi_t(0)\|_{\dot{H}^{1}( \R^5)}+\|\Phi_{tt}(0)\|_{L^2(\R^5)}\lesssim 1,
\label{phit0-h1}
\end{equation}
appearing in the proof of Proposition \ref{prop-Phit-h1},
\begin{equation}
 \|\Phi_{tt}(0)\|_{\dot{H}^{1}( \R^5)}+\|\Phi_{ttt}(0)\|_{L^2(\R^5)}\lesssim 1,
\label{phitt0-h1}
\end{equation}
intrinsically featured in the argument for Proposition \ref{prop-Phitt-h1}, and
\begin{equation}
 \|\Phi_t(0)\|_{\dot{H}^{s-1}( \R^5)}+\|\Phi_{tt}(0)\|_{\dot{H}^{s-2}(\R^5)}\lesssim 1,
\label{phit0-hs}
\end{equation}
which is part of Proposition \ref{prop-phit}.

To start with, we recall \eqref{Phit-final} and \eqref{Phitt-final}, i.e., 
\begin{equation}
\Phi_{t}=\tA^{1/2}(r,v)\,v_{t},\qquad
\Phi_{tt}=\tA^{1/2}(r,v)\,v_{tt}+ \tA^{-1/2}(r,v)\,\frac{\sin(2u)}{r}\,v^2_{t},
\label{phit-phitt}
\end{equation}
and we compute, on the basis of the latter,
\begin{equation}
\aligned
\Phi_{ttt}=\ &\tA^{1/2}(r,v)\,v_{ttt}+ 3\tA^{-1/2}(r,v)\,\frac{\sin(2u)}{r}\,v_{tt}\,v_{t}\\
&+\tA^{-3/2}(r,v)\left(2\cos(2u)-4\,\frac{\sin^4(u)}{r^2}\right)v^3_{t}.
\endaligned
\label{phittt}
\end{equation}
Additionally, we notice that the hypothesis of Theorem \ref{main-th-v-2} on the initial data, i.e.,
\begin{equation}
v(0)\in H^s(\R^5), \qquad v_t(0)\in H^{s-1}(\R^5), \qquad s>\frac{7}{2},
\label{v0-hs}
\end{equation}
together with the classical Sobolev embedding \eqref{Sob-classic}, yields
\begin{equation}
v(0)\in H^{1,\infty}(\R^5), \qquad v_t(0)\in L^{\infty}(\R^5),
\label{dv0-li}
\end{equation}
a fact that will be used extensively in what follows. With these prerequisites out of the way, we can first prove \eqref{phit0-h1}. 

\begin{customthm}{A.1}
Under the assumptions of Theorem \ref{main-th-v-2}, the estimate \eqref{phit0-h1} is valid.
\end{customthm}
\begin{proof}
Given that 
\[
\|\Phi_t(0)\|_{\dot{H}^{1}( \R^5)}\sim \|\partial_r\Phi_t(0)\|_{L^{2}( \R^5)}, 
\]
we rely on \eqref{phit-phitt} to initially calculate 
\begin{equation*}
\partial_r\Phi_t= \tA^{1/2}(r,v)\,\partial_rv_{t}+\frac{1}{2}\tA^{-1/2}(r,v)\left(-\frac{4\sin^2(u)}{r^3}+\frac{2\sin(2u)}{r^2}u_r\right)v_{t},
\end{equation*}
which is analyzed separately in the $r\leq 1$ and $r>1$ regimes. In the former, we easily have
\begin{equation}
1\leq \tA(r,v)\lesssim 1+v^2,
\label{ta-l1}
\end{equation}
while a Maclaurin analysis yields
\begin{equation}
\left|-\frac{4\sin^2(u)}{r^3}+\frac{2\sin(2u)}{r^2}u_r\right|\lesssim |v_r||v| +rv^4.
\label{sin2u-r3-l1}
\end{equation}
When $r>1$, it follows directly that
\begin{equation}
\tA(r,v)\sim 1
\label{ta-g1}
\end{equation}
and 
\begin{equation}
\left|-\frac{4\sin^2(u)}{r^3}+\frac{2\sin(2u)}{r^2}u_r\right|\lesssim \frac{1}{r^3}+\frac{|u_r|}{r^2}\lesssim \frac{1}{r^3}+\frac{|v|+r|v_r|+|\varphi_r|}{r^2}.
\label{sin2u-r3-g1}
\end{equation}
Based on these findings, \eqref{v0-hs}, and \eqref{dv0-li}, we infer that
\begin{equation*}
\aligned
\|\partial_r\Phi_t(0)\|_{L^{2}( \R^5)}&\lesssim \left(1+\|v(0)\|_{L^{\infty}( \R^5)}\right)\|\partial_rv_t(0)\|_{L^{2}( \R^5)}\\
&\quad+\big(\|v_r(0)\|_{L^{\infty}( \R^5)}\|v(0)\|_{L^{\infty}( \R^5)}+\|v(0)\|^4_{L^{\infty}( \R^5)}\\
&\qquad\ +1+\|v_r(0)\|_{L^{\infty}( \R^5)}\big)\|v_t(0)\|_{L^{2}( \R^5)}\\
&\lesssim 1.
\endaligned
\end{equation*}

Thus, we are left to show that
\begin{equation}
\|\Phi_{tt}(0)\|_{L^2(\R^5)}\lesssim 1,
\label{Phitt-0}
\end{equation}
and, for this purpose, we first notice that
\begin{equation}
\frac{|\sin(2u)|}{r}\lesssim\begin{cases}
|v|, \quad r\leq 1,\\
\frac{1}{r}, \quad \ \, r >1.
\end{cases}
\label{sin2u}
\end{equation}
Due to this estimate and the ones used in the argument for $\Phi_t(0)$, we deduce
\begin{equation}
\aligned
\|\Phi_{tt}(0)\|_{L^2(\R^5)}&\lesssim  \left(1+\|v(0)\|_{L^{\infty}( \R^5)}\right)\|v_{tt}(0)\|_{L^{2}( \R^5)}\\
&\quad+ \left(1+\|v(0)\|_{L^{\infty}( \R^5)}\right)\|v_{t}(0)\|_{L^{\infty}( \R^5)}\|v_{t}(0)\|_{L^{2}( \R^5)}\\
&\lesssim \|v_{tt}(0)\|_{L^{2}( \R^5)}+1\\
&\leq \|\Box v(0)\|_{L^{2}( \R^5)}+\|\Delta v(0)\|_{L^{2}( \R^5)}+1\\
&\lesssim \|\Box v(0)\|_{L^{2}( \R^5)}+1.
\endaligned
\label{ptt-Dv}
\end{equation}

It may seem right now that we could just invoke \eqref{vtt-dv} and, consequently,  \eqref{Phitt-0} would follow. However, this is not the case as \eqref{phit0-h1} is required in the main argument at a point that precedes and most likely influences the proof of \eqref{vtt-dv}, making this approach circular. However, parts of the asymptotics used in proving \eqref{vtt-dv} can still be employed here, as they were argued for with facts found prior to \eqref{phit0-h1}. A term-by-term analysis of the right-hand side of \eqref{main-v} evaluated at $t=0$ yields:
\begin{equation*}
\left\| \frac{1}{r}\Delta_3 \varphi\right\|_{L^2(\R^5)}+\left\|\frac{2}{r^2}\varphi_{>1} v(0)\right\|_{L^2(\R^5)} \lesssim \left\|\frac{1}{r}\right\|_{L^2( \{1\leq r\leq 2\})}+\|v(0)\|_{L^2(\R^5)}\lesssim 1,
\end{equation*}
\begin{equation*}
\aligned
&\left\|\varphi_{<1}\left(\frac{1}{r}N(r, rv(0), \nabla(rv)(0))+\frac{2}{r^2}v(0)\right)\right\|_{L^2(\R^5)}\\
&\qquad\lesssim\left\| |\varphi_{<1}|\left(|v(0)|^3+|v(0)|^5+|v(0)||\nabla v(0)|^2+rv^4(0)|v_r(0)|\right)\right\|_{L^2(\R^5)}\\
&\qquad\lesssim \bigg(\|v(0)\|^2_{L^{\infty}( \R^5)}+\|v(0)\|^4_{L^{\infty}( \R^5)}+\|\nabla v(0)\|^2_{L^{\infty}( \R^5)}\\
&\qquad\qquad+\|v(0)\|^3_{L^{\infty}( \R^5)}\|\nabla v(0)\|_{L^{\infty}( \R^5)}\bigg)\|v(0)\|_{L^{2}( \R^5)}\\
&\qquad\lesssim 1,
\endaligned
\end{equation*}
and
\begin{equation*}
\aligned
&\left\|\frac{1}{r}\varphi_{>1}N\left(r, rv+\varphi, \nabla(rv+\varphi)\right)\big|_{t=0}\right\|_{L^2(\R^5)}\\
&\quad\lesssim \left\||\varphi_{>1}|\left(\frac{1}{r^3}+\frac{|\nabla v(0)|^2}{r}+\frac{|v_r(0)|}{r^3}+\frac{|v(0)|^2}{r^3}+\frac{|v(0)|}{r^4}\right)\right\|_{L^2(\R^5)}\\&\quad\lesssim 1+\left(\|\nabla v(0)\|_{L^{\infty}( \R^5)}+1\right)\|\nabla v(0)\|_{L^{2}( \R^5)}+\left(\|v(0)\|_{L^{\infty}( \R^5)}+1\right)\|v(0)\|_{L^{2}( \R^5)}\\
&\quad\lesssim 1.
\endaligned
\end{equation*}
Collectively, these three estimates and \eqref{main-v} show that
\begin{equation}
\|\Box v(0)\|_{L^{2}( \R^5)}\lesssim 1,
\label{boxv-0} 
\end{equation}
which, on the basis of \eqref{ptt-Dv}, implies that \eqref{Phitt-0} holds true. This finishes the proof of this proposition.
\end{proof}

Next, we can build off of this result and prove that \eqref{phitt0-h1} is valid.

\begin{customthm}{A.2}
Under the assumptions of Theorem \ref{main-th-v-2}, the estimate \eqref{phitt0-h1} holds true.
\end{customthm}
\begin{proof}
We start by reducing the argument for \eqref{phitt0-h1} to showing that
\begin{equation}
\|\nabla\Box v(0)\|_{L^{2}( \R^5)}\lesssim 1.
\label{nabla-boxv-0}
\end{equation}
In the case of the $L^2$ norm for $\Phi_{ttt}(0)$, due to \eqref{phittt}, we can directly estimate as follows:
\begin{equation*}
\aligned
\|\Phi_{ttt}(0)\|&_{L^{2}( \R^5)}\\
&\lesssim \|\tA^{1/2}(r,v(0))\|_{L^{\infty}( \R^5)}\|v_{ttt}(0)\|_{L^{2}( \R^5)}\\
&\quad + \left\|\frac{\sin(2u(0))}{r}\right\|_{L^{\infty}( \R^5)}\|v(0)\|_{L^{\infty}( \R^5)}\|v_{tt}(0)\|_{L^{2}( \R^5)}\\
&\quad +\left\|2\cos(2u(0))-4\,\frac{\sin^4(u(0))}{r^2}\right\|_{L^{\infty}( \R^5)}\|v_{t}(0)\|^2_{L^{\infty}( \R^5)}\|v_{t}(0)\|_{L^{2}( \R^5)}.
\endaligned
\end{equation*}
From the previous proposition (i.e., \eqref{ta-l1}, \eqref{ta-g1}, \eqref{sin2u-r3-l1}, \eqref{sin2u-r3-g1}) we already know that
\begin{equation}
\|\tA^{1/2}(r,v(0))\|_{L^{\infty}( \R^5)}+\left\|\frac{\sin(2u(0))}{r}\right\|_{L^{\infty}( \R^5)}\lesssim 1+\|v(0)\|_{L^{\infty}( \R^5)}
\label{ta-sin2u}
\end{equation}
and 
\begin{equation}
\|v_{tt}(0)\|_{L^{2}( \R^5)}\lesssim 1.
\label{vtt-0}
\end{equation}
Moreover, it is easy to see that
\begin{equation*}
\left|2\cos(2u)-4\,\frac{\sin^4(u)}{r^2}\right|\lesssim\begin{cases}
1+r^2v^4, \quad r\leq 1,\\
1, \quad \qquad \quad\, r >1,
\end{cases}
\end{equation*}
and, consequently, 
\begin{equation*}
\left\|2\cos(2u(0))-4\,\frac{\sin^4(u(0))}{r^2}\right\|_{L^{\infty}( \R^5)}\lesssim 1+\|v(0)\|^4_{L^{\infty}( \R^5)}.
\end{equation*}
Therefore, we can infer from these bounds, \eqref{v0-hs}, and \eqref{dv0-li} that
\begin{equation}
\aligned
\|\Phi_{ttt}(0)\|_{L^{2}( \R^5)}&\lesssim \|v_{ttt}(0)\|_{L^{2}( \R^5)}+1\\
&\leq\|\Delta v_{tt}(0)\|_{L^{2}( \R^5)}+\|\partial_t \Box v(0)\|_{L^{2}( \R^5)}+1\\
&\lesssim \|\partial_t \Box v(0)\|_{L^{2}( \R^5)}+1.
\endaligned
\label{dt-boxv-0}
\end{equation}

For the $\dot{H}^1$ norm of $\Phi_{tt}(0)$, as
\[
\|\Phi_{tt}(0)\|_{\dot{H}^{1}( \R^5)}\sim \|\partial_r\Phi_{tt}(0)\|_{L^{2}( \R^5)}, 
\]
we use \eqref{phit-phitt} to compute
\begin{equation*}
\aligned
\partial_r\Phi_{tt}&=\tA^{1/2}(r,v)\,\partial_rv_{tt}\\
&+\frac{1}{2}\left(-\frac{4\sin^2(u)}{r^3}+\frac{2\sin(2u)}{r^2}u_r\right)\left(\tA^{-1/2}(r,v)v_{tt}-\tA^{-3/2}(r,v)\frac{\sin(2u)}{r}v_t^2\right)\\
&+\tA^{-1/2}(r,v)\left(-\frac{\sin(2u)}{r^2}+\frac{2\cos(2u)}{r}u_r\right)v^2_{t}\\
&+2\tA^{-1/2}(r,v)\,\frac{\sin(2u)}{r}\,\partial_r v_t\,v_{t}.
\endaligned
\end{equation*}
In addition to the terms controlled by \eqref{ta-sin2u}, the proof of the previous proposition implies through \eqref{sin2u-r3-l1} and \eqref{sin2u-r3-g1} the estimate 
\begin{equation*}
\aligned
&\left\|-\frac{4\sin^2(u(0))}{r^3}+\frac{2\sin(2u(0))}{r^2}u_r(0)\right\|_{L^\infty(\R^5)}\\
&\qquad\qquad \lesssim\|v_r(0)\|_{L^{\infty}( \R^5)}\|v(0)\|_{L^{\infty}( \R^5)}+\|v(0)\|^4_{L^{\infty}( \R^5)}+1+\|v_r(0)\|_{L^{\infty}( \R^5)}.
\endaligned
\end{equation*}
A similar analysis yields in a straightforward way 
\begin{equation*}
\left|-\frac{\sin(2u)}{r^2}+\frac{2\cos(2u)}{r}u_r\right|\lesssim\begin{cases}
|v_r|+r|v|^3, \qquad\qquad\quad\, r\leq 1,\\
\frac{1}{r^2}+\frac{|v|+r|v_r|+|\varphi_r|}{r}, \quad  \quad\, r >1,
\end{cases}
\end{equation*}
and, subsequently, 
\begin{equation*}
\left\|-\frac{\sin(2u(0))}{r^2}+\frac{2\cos(2u(0))}{r}u_r(0)\right\|_{L^{\infty}( \R^5)}\lesssim \|v_r(0)\|_{L^{\infty}( \R^5)}+\|v(0)\|^3_{L^{\infty}( \R^5)}+1.
\end{equation*}
Hence, collecting all these facts and relying again on \eqref{v0-hs}, \eqref{dv0-li}, and \eqref{vtt-0}, we deduce that
\begin{equation*}
\aligned
&\|\partial_r\Phi_{tt}(0)\|_{L^{2}( \R^5)}\\
&\qquad\lesssim (1+\|v(0)\|_{L^{\infty}( \R^5)})\|\partial_rv_{tt}(0)\|_{L^2(\R^5)}\\
&\quad\qquad+\left(\|v_r(0)\|_{L^{\infty}( \R^5)}\|v(0)\|_{L^{\infty}( \R^5)}+\|v(0)\|^4_{L^{\infty}( \R^5)}+1+\|v_r(0)\|_{L^{\infty}( \R^5)}\right)\\
&\quad\qquad\quad \cdot\left\{\|v_{tt}(0)\|_{L^2(\R^5)}+(1+\|v(0)\|_{L^{\infty}( \R^5)})\|v_t(0)\|_{L^{\infty}( \R^5)}\|v_t(0)\|_{L^2( \R^5)}\right\}\\
&\quad\qquad+\left(\|v_r(0)\|_{L^{\infty}( \R^5)}+\|v(0)\|^3_{L^{\infty}( \R^5)}+1\right)\|v_t(0)\|_{L^{\infty}( \R^5)}\|v_t(0)\|_{L^2( \R^5)}\\
&\quad\qquad+(1+\|v(0)\|_{L^{\infty}( \R^5)})\|v_t(0)\|_{L^{\infty}( \R^5)}\|\partial_rv_{t}(0)\|_{L^2(\R^5)}\\
&\qquad\lesssim \|\partial_rv_{tt}(0)\|_{L^2(\R^5)}+1\\
&\qquad\leq \|\partial_r\Delta v(0)\|_{L^2(\R^5)}+\|\partial_r\Box v(0)\|_{L^2(\R^5)}+1\\
&\qquad\lesssim \|\partial_r\Box v(0)\|_{L^2(\R^5)}+1.
\endaligned
\end{equation*}
Together with \eqref{dt-boxv-0}, this inequality finishes the argument claiming that, in order to prove \eqref{phitt0-h1}, it is sufficient to show \eqref{nabla-boxv-0}.

For proving \eqref{nabla-boxv-0}, we adopt a similar approach to the one leading to \eqref{boxv-0}, in the sense that we estimate the gradient $\nabla_{t,r}=(\partial_t, \partial_r)$ at $t=0$ for each of the terms on the right-hand side of \eqref{main-v}. In fact, one can see relatively easily that most of the corresponding terms in the two analyses share a generic core and, thus, we can just investigate the slight differences appearing in this new framework. First, it is straightforward to argue that
\begin{equation*}
\aligned
&\left\|\nabla_{t,r}\left( \frac{1}{r}\Delta_3 \varphi\right)\right\|_{L^2(\R^5)}+\left\|\nabla_{t,r}\left(\frac{2}{r^2}\varphi_{>1} v\right)(0)\right\|_{L^2(\R^5)}\\ 
&\qquad\qquad\qquad\lesssim 1+\|v(0)\|_{L^2(\R^5)}+\|\nabla_{t,r}v(0)\|_{L^2(\R^5)}\\
&\qquad\qquad\qquad\lesssim 1.
\endaligned
\end{equation*}

Secondly, when we deal with terms involving the cutoff $\varphi_{<1}$, we notice that differentiation of the expressions having the generic profile $\tilde{N}(rv)v^k$ is easy to manage. This is due to the control \eqref{dni-li} we have on $\tilde{N}$ and that the gradient of such an expression is made of terms like
\[
\tilde{N}(rv)v^{k-1}\nabla_{t,r}v, \qquad \tilde{N}'(rv)v^{k}r\nabla_{t,r}v, \qquad \tilde{N}'(rv)v^{k+1}.
\]
Therefore, by comparison to the analysis for \eqref{boxv-0}, we either replace $v(0)$ by $\nabla_{t,r}v(0)$ or we have an extra factor of $r\nabla_{t,r}v(0)$ or $v(0)$, For the former case, we estimate the gradient in the same $L^p$ space (i.e., $L^\infty(\R^5)$ or $L^2(\R^5)$) as we did $v(0)$, while for the latter one we can place both extra factors in  $L^\infty(\R^5)$ due to \eqref{dv0-li} and to the presence of $\varphi_{<1}$, which forces $r\leq 1$. Similar arguments can be done for the terms $N_3(rv)v(v_t^2-v_r^2)$ and $N_4(rv)rv^4v_r$, with slight modifications for when the gradient falls on the derivative terms. In this case, we need to estimate
\[
N_3(rv(0))v(0)(v_t(0)\nabla_{t,r}v_t(0)-v_r(0)\nabla_{t,r}v_r(0))
\]
and
\[
N_4(rv(0))rv^4(0)\nabla_{t,r}v_r(0),
\]
and we place all factors in $L^\infty(\R^5)$, with the exception of the second order derivatives, which are bounded in $L^2(\R^5)$. We control $v_{tt}(0)$ through \eqref{vtt-0} and 
\[
\|\partial_tv_r(0)\|_{L^2(\R^5)}=\|\partial_rv_t(0))\|_{L^2(\R^5)}\sim \|v_t(0))\|_{\dot{H}^1(\R^5)}\lesssim 1.
\]
For $v_{rr}(0)$, we argue using \eqref{rad-Sob-2} and \eqref{dv0-li} to infer that
\begin{equation*}
\aligned
\|v_{rr}(0))\|_{L^2(\R^5)}&\lesssim \|\Delta v(0))\|_{L^2(\R^5)}+\left\|\frac{v_{r}(0)}{r}\right\|_{L^2(\R^5)}\\ &\lesssim \|v(0))\|_{\dot{H}^2(\R^5)}+\left\|\frac{1}{(1+r^2)r}\right\|_{L^2(\R^5)}\\
&\lesssim 1.
\endaligned
\end{equation*}
This finishes the discussion of terms localized by $\varphi_{<1}$.

Finally, we address the gradient for the terms on the right-hand side of \eqref{main-v} involving $N(r,rv+\varphi, \nabla(rv+\varphi))$. We claim the analysis is almost equivalent to the one just above, with one exception. In this case, factors of $r$ introduced by differentiation are not friendly due to the localization induced by $\varphi_{>1}$. However, we claim that in the structure of $N(r,rv+\varphi, \nabla(rv+\varphi))$, there are sufficient negative powers of $r$ to offset this issue and we ask the careful reader to verify this.
\end{proof}

In conclusion of this appendix, we show that \eqref{phit0-hs} holds true.

\begin{customthm}{A.3}
Under the assumptions of Theorem \ref{main-th-v-2}, the estimate \eqref{phit0-hs} is valid.
\end{customthm}
\begin{proof}
We first focus on the finiteness of the $\dot{H}^{s-1}(\R^5)$ norm and, by applying the fractional Leibniz estimate \eqref{Lbnz-0} in the context of \eqref{phit-phitt}, we deduce
\begin{equation*}
\aligned
\|\Phi_t(0)\|_{\dot{H}^{s-1}( \R^5)}\lesssim\, &\|\tA^{1/2}(r,v(0))\|_{\dot{H}^{s-1}( \R^5)}\|v_t(0)\|_{L^\infty( \R^5)}\\ &+ \|\tA^{1/2}(r,v(0))\|_{L^\infty( \R^5)}\|v_t(0)\|_{\dot{H}^{s-1}( \R^5)}.
\endaligned
\end{equation*}
If we rely on \eqref{v0-hs}, \eqref{dv0-li}, and \eqref{ta-sin2u}, it follows that
\begin{equation}
\|\Phi_t(0)\|_{\dot{H}^{s-1}( \R^5)}\lesssim \|\tA^{1/2}(r,v(0))\|_{\dot{H}^{s-1}( \R^5)}+1.
\label{phiths-1}
\end{equation}
Next, we use the Moser inequality \eqref{Moser} for the $C^\infty$ function
\[
F:\R\to \R, \qquad F(x)=(1+2x^2)^{1/2}-1,
\]
to derive that
\begin{equation}
\aligned
 &\|\tA^{1/2}(r,v(0))\|_{\dot{H}^{s-1}( \R^5)}\\
&\qquad\qquad\lesssim  \|\tA^{1/2}(r,v(0))-1\|_{H^{s-1}( \R^5)}\\
&\qquad\qquad\lesssim \gamma\left(\left\|\frac{\sin(u(0))}{r}\right\|_{L^\infty(\R^5)}\right)\left\|\frac{\sin(u(0))}{r}\right\|_{H^{s-1}(\R^5)}\\
 &\qquad\qquad\lesssim \gamma\left(\left\|\frac{\sin(u(0))}{r}\right\|_{L^\infty(\R^5)}\right)\bigg(\left\|\frac{\sin(u(0))}{r}\right\|_{H^{s-1}(\{1\leq r\leq 2\})}\\
 &\qquad\qquad\qquad\qquad\qquad\qquad\qquad\qquad+\left\|\frac{\sin(rv(0))}{r}\right\|_{H^{s-1}(\R^5)}\bigg)\\
 &\qquad\qquad\lesssim \left\|\frac{\sin(rv(0))}{r}\right\|_{H^{s-1}(\R^5)}+1,
\endaligned
\label{phiths-2}
\end{equation}
where the last line is motivated by an argument identical to the one producing \eqref{ta-sin2u}.

Due to $s>\frac{7}{2}$, $H^{s-1}(\R^5)$ is an algebra and, also factoring in \eqref{v0-hs} and \eqref{Moser}, we obtain
\begin{equation}
\aligned
\left\|\frac{\sin(rv(0))}{r}\right\|&_{H^{s-1}(\R^5)}\\
&\leq \|v(0)\|_{H^{s-1}( \R^5)}\left(\left\|\frac{\sin(rv(0))-rv(0)}{rv(0)}\right\|_{H^{s-1}(\R^5)}+1\right)\\
&\lesssim \left\|\frac{\sin(rv(0))-rv(0)}{rv(0)}\right\|_{H^{s-1}(\R^5)}+1\\
&\lesssim \gamma(\|r^2v^2(0)\|_{L^\infty( \R^5)})\,\|r^2v^2(0)\|_{H^{s-1}( \R^5)}+1,
\endaligned
\label{phiths-3}
\end{equation}
since 
\[
\frac{\sin(x)-x}{x}=H(x^2)
\]
for a function $H\in C^\infty(\R;\R)$ with $H(0)=0$. However, as $v(0)\in H^1({\R^5})$, we can infer according to \eqref{rad-Sob-2} that
\begin{equation*}
\|r^2v(0)\|_{L^\infty( \R^5)}+\|D^{s-1}r\, r\,v(0)\|_{L^\infty( \R^5)}\lesssim 1,
\end{equation*}
which, jointly with \eqref{Lbnz-0}, \eqref{v0-hs}, and \eqref{dv0-li}, yields
\[
\|r^2v^2(0)\|_{L^\infty( \R^5)}\lesssim 1
\]
and
\[
\aligned
\|r^2v^2(0)\|_{H^{s-1}( \R^5)}&\sim \|r^2v^2(0)\|_{L^2( \R^5)}+\|D^{s-1}\left(r^2v^2(0)\right)\|_{L^{2}( \R^5)}\\
&\lesssim\|r^2v(0)\|_{L^\infty( \R^5)}\|v(0)\|_{L^2( \R^5)}+\|r^2v(0)\|_{L^\infty( \R^5)}\|v(0)\|_{\dot{H}^{s-1}( \R^5)}\\
&\quad+\|D^{s-1}r\, r\,v(0)\|_{L^\infty( \R^5)}\|v(0)\|_{L^2( \R^5)}\\
&\lesssim 1.
\endaligned
\]
Therefore, by also relying on \eqref{phiths-1}-\eqref{phiths-3}, we have that
\begin{equation}
\|\Phi_t(0)\|_{\dot{H}^{s-1}( \R^5)}\lesssim 1.
\label{phit0-hs-1}
\end{equation}

In order to finish the argument, we need to prove the finiteness of the $\dot{H}^{s-2}( \R^5)$ norm in \eqref{phit0-hs} and, for this purpose, we start by analyzing the second term on the right-hand side of \eqref{phit-phitt}. An application of \eqref{Lbnz-0} coupled with \eqref{v0-hs}, \eqref{dv0-li}, \eqref{ta-sin2u}, and the obvious bound
\[
0<\tA^{-1/2}(r,v(0))\leq 1,
\]
yields 
\begin{equation*}
\aligned
&\left\|\tA^{-1/2}(r,v(0))\,\frac{\sin(2u(0))}{r}\,v^2_{t}(0)\right\|_{\dot{H}^{s-2}( \R^5)}\\
&\qquad\lesssim \left\|\tA^{-1/2}(r,v(0))\right\|_{\dot{H}^{s-2}( \R^5)}\left\|\frac{\sin(2u(0))}{r}\right\|_{L^\infty(\R^5)}\left\|v_{t}(0)\right\|^2_{L^\infty( \R^5)}\\
&\qquad\quad+\left\|\tA^{-1/2}(r,v(0))\right\|_{L^\infty( \R^5)}\left\|\frac{\sin(2u(0))}{r}\right\|_{\dot{H}^{s-2}( \R^5)}\left\|v_{t}(0)\right\|^2_{L^\infty( \R^5)}\\
&\qquad\quad+\left\|\tA^{-1/2}(r,v(0))\right\|_{L^\infty( \R^5)}\left\|\frac{\sin(2u(0))}{r}\right\|_{L^\infty( \R^5)}\left\|v_{t}(0)\right\|_{L^{\infty}( \R^5)}\left\|v_{t}(0)\right\|_{\dot{H}^{s-2}( \R^5)}\\
&\qquad\lesssim \left\|\tA^{-1/2}(r,v(0))\right\|_{\dot{H}^{s-2}( \R^5)}+\left\|\frac{\sin(2u(0))}{r}\right\|_{\dot{H}^{s-2}( \R^5)}+1.
\endaligned
\end{equation*}
However, an approach identical to the one producing \eqref{phit0-hs-1} leads to 
\[
\left\|\tA^{-1/2}(r,v(0))\right\|_{H^{s-1}( \R^5)}+\left\|\frac{\sin(2u(0))}{r}\right\|_{H^{s-1}( \R^5)}\lesssim 1
\]
and, subsequently, 
\[
\left\|\tA^{-1/2}(r,v(0))\,\frac{\sin(2u(0))}{r}\,v^2_{t}(0)\right\|_{\dot{H}^{s-2}( \R^5)}\lesssim 1.
\]
Lastly, we investigate the first term on the right-hand side of \eqref{phit-phitt}, and yet another usage of  \eqref{Lbnz-0} jointly with the Sobolev embeddings \eqref{Sob-gen}, \eqref{ta-sin2u}, \eqref{vtt-0}, and the analysis deriving \eqref{phit0-hs-1} implies
\begin{equation*}
\aligned
\left\|\tA^{1/2}(r,v(0))\,v_{tt}(0)\right\|_{\dot{H}^{s-2}( \R^5)}&\lesssim \left\|\tA^{1/2}(r,v(0))\right\|_{\dot{H}^{s-2, 10/3}( \R^5)}\left\|v_{tt}(0)\right\|_{L^5( \R^5)}\\
&\quad+\left\|\tA^{1/2}(r,v(0))\right\|_{L^\infty( \R^5)}\left\|v_{tt}(0)\right\|_{\dot{H}^{s-2}( \R^5)}\\
&\lesssim \left\|\tA^{1/2}(r,v(0))\right\|_{\dot{H}^{s-1}( \R^5)}\left\|v_{tt}(0)\right\|_{H^{s-2}( \R^5)}\\
&\quad+\left\|v_{tt}(0)\right\|_{\dot{H}^{s-2}( \R^5)}\\
&\lesssim \left\|v_{tt}(0)\right\|_{\dot{H}^{s-2}( \R^5)}+1.
\endaligned
\end{equation*}
Now, we rely on \eqref{v0-hs} to infer
\[
\aligned
 \left\|v_{tt}(0)\right\|_{\dot{H}^{s-2}( \R^5)}&\lesssim \left\|\Delta v(0)\right\|_{\dot{H}^{s-2}( \R^5)}+\left\|\Box v(0)\right\|_{\dot{H}^{s-2}( \R^5)}\\
 &\lesssim \left\|v(0)\right\|_{\dot{H}^{s}( \R^5)}+\left\|\Box v(0)\right\|_{\dot{H}^{s-2}( \R^5)}\\
 &\lesssim \left\|\Box v(0)\right\|_{\dot{H}^{s-2}( \R^5)}+1
 \endaligned
\]
and we argue that, following the blueprint of the analysis for the $\dot{H}^{s-1}( \R^5)$, one also obtains
\[
\left\|\Box v(0)\right\|_{\dot{H}^{s-2}( \R^5)}\lesssim 1.
\]
We let the avid reader fill in the details. The end result is that the last four estimates together give
\begin{equation}
\|\Phi_{tt}(0)\|_{\dot{H}^{s-2}( \R^5)}\lesssim 1
\label{phit0-hs-2}
\end{equation}
and the proof of \eqref{phit0-hs} is finished.
\end{proof}

\bibliographystyle{amsplain}
\bibliography{anwb-recent}

\end{document}